%% file: likely_path.tex
\documentclass[11pt]{article}
\usepackage{xcolor}
\usepackage{amsmath,amssymb,amsthm,mathtools, hyperref}
\hypersetup{colorlinks=true}
\usepackage{enumitem}
\usepackage{tikz}
\usetikzlibrary{decorations.pathmorphing}
\usepackage{bbm}
\usepackage{stmaryrd}
\usepackage{caption}

\usepackage[style=alphabetic, maxbibnames=9]{biblatex}
\DeclareSourcemap{
  \maps[datatype=bibtex]{
    \map{
      \step[fieldsource=doi, final]
      \step[fieldset=url, null]
      \step[fieldset=urldate, null]
    }
  }
}

\newcommand{\disteq}{\ensuremath{\stackrel{d}{=}}}

\newcommand{\cornerpath}{\gamma^{\raisebox{-0.5ex}{$\scriptstyle\ulcorner$}}}
\newcommand{\rcornerpath}{\gamma^{\raisebox{-0.5ex}{$\scriptstyle\lrcorner$}}}

\newcommand{\bone}{\mathbbm{1}}

\newcommand{\bltriangle}{\mathord{\text{%
\tikz[baseline=0.1ex, line join=round] {%
\draw[line width=0.06em] (0,0) -- (1ex,0) -- (0,1ex) -- cycle;%
}%
}}}

\newcommand{\tltriangle}{\mathord{\text{%
\tikz[baseline=0.1ex, line join=round] {%
\draw[line width=0.06em] (0,0) -- (0,1ex) -- (1ex,1ex) -- cycle;%
}%
}}}

\newcommand{\trtriangle}{\mathord{\text{%
\tikz[baseline=0.1ex, line join=round] {%
\draw[line width=0.06em] (0,1ex) -- (1ex,1ex) -- (1ex,0) -- cycle;%
}%
}}}

\newcommand{\brtriangle}{\mathord{\text{%
\tikz[baseline=0.1ex, line join=round] {%
\draw[line width=0.06em] (0,0) -- (1ex,0) -- (1ex,1ex) -- cycle;%
}%
}}}

\newcommand{\PP}{\mathbb{P}}
\newcommand{\EE}{\mathbb{E}}
\newcommand{\ZZ}{\mathbb{Z}}
\newcommand{\RR}{\mathbb{R}}
\newcommand{\1}{\mathbf{1}}
\newcommand{\eps}{\epsilon}
\DeclareMathOperator*{\argmax}{arg\,max}
\DeclareMathOperator*{\esssup}{ess\,sup}
\DeclareMathOperator{\sgn}{sgn}

\DeclarePairedDelimiter\abs{\lvert}{\rvert}
\DeclarePairedDelimiter\norm{\lVert}{\rVert}
\DeclarePairedDelimiter\brac{\lparen}{\rparen}
\DeclarePairedDelimiter\sqrbrac{[}{]}

\DeclarePairedDelimiter\set{\{}{\}}
\DeclarePairedDelimiter\floor{\lfloor}{\rfloor}
\DeclarePairedDelimiter\ceil{\lceil}{\rceil}
\DeclarePairedDelimiter\bbrac{\llbracket}{\rrbracket}

\newtheorem{theorem}{Theorem}
\newtheorem{proposition}{Proposition}
\newtheorem{corollary}{Corollary}
\newtheorem{lemma}{Lemma}
\newtheorem{assumption}{Assumption}

\theoremstyle{remark}
\newtheorem{remark}{Remark}

\newcommand{\Exp}{\ensuremath{\mathrm{Exp}}}

\newcommand{\LOE}{\ensuremath{\mathsf{LOE}}}
\newcommand{\LUE}{\ensuremath{\mathsf{LUE}}}
\newcommand{\LSE}{\ensuremath{\mathsf{LSE}}}

\DeclarePairedDelimiterXPP\Ex[1]{\mathbb{E}}[]{}{

#1}
\DeclarePairedDelimiterXPP\Exabs[1]{\mathbb{E}}\lvert\rvert{}{

#1}
\DeclarePairedDelimiterXPP\Prob[1]{\mathbb{P}}(){}{

#1}
\DeclarePairedDelimiterXPP\Var[1]{\mathrm{Var}}\lparen\rparen{}{

#1}

\title{On the most likely geodesic in last passage percolation}
\author{Sam McKeown \and Shuta Nakajima}

\begin{document}
\maketitle
\begin{abstract}
    We consider the problem of identifying the paths most likely to occur as geodesics in last passage percolation. Heuristics suggest that these modal paths should be the extreme \emph{corner paths}, going straight between the corners of the cube of accessible vertices. We identify three mechanisms which favour such corner paths, and show that they are more likely to appear than all but a vanishingly small portion of paths. We make more specific comparisons in exponential last passage percolation, where moderate deviation estimates may be used to show that corner paths are nearly modal in a precise sense. Finally, we show a form of monotonicity of geodesic probabilities in a special case and conjecture that this holds in general.
\end{abstract}
\section{Introduction}
Consider last passage percolation (LPP) on $\ZZ^d$ with i.i.d.\ weights
$\omega=(\omega_x)_{x\in \ZZ^d}$.
A directed path on the lattice is a sequence $\gamma = (\gamma^1, \dots, \gamma^N)$ such that $\gamma^{i+1} - \gamma^i  \in \set{e_1,\dots,e_d}$. To such a path, associate a passage time
\begin{equation*}
    L(\gamma) = \sum_{x \in \gamma} \omega_x.
\end{equation*}
Let $\Gamma_n$ be the set of directed paths connecting $0$ to $(n, \dots, n) \in \ZZ^d$, and define the passage time to $(n, \dots, n)$ as
\begin{equation}
\label{eq:LPPDef}
    L_n = \max_{\gamma \in \Gamma_n}\sum_{x \in \gamma} \omega_x.
\end{equation}
Write $\Gamma_n^* = \{\gamma \in \Gamma_n : L(\gamma) = L_n\}$ for the (non-empty) set of paths at which the maximum in \eqref{eq:LPPDef} is attained. Call an element $\gamma_n^* \in \Gamma_n^*$ a \emph{geodesic}. Observe that geodesics will be unique almost surely when the distribution of $\omega_x$ is continuous.

It is quite unlikely for a given path to be a geodesic. Consider the time constant $\mu$, defined by the almost sure limit
\[
\mu = \lim_{n \to \infty} \frac{L_n}{d n}.
\]
This represents the average weight seen along a geodesic. Except in trivial cases, $\mu$ will be greater than the mean of the weights, and thus attaining this average along a particular path is an exponentially unlikely event. It would be of interest to understand better the probability of this exponentially unlikely event, and especially its dependence on the geometry of the path. 

As a first step in this direction, the authors of \cite{AC:lpp-geom} asked: \emph{what is the most likely shape of $\gamma^*_n$?} That is, they asked for the value of
\[
\argmax_{\gamma \in \Gamma_n} \PP(\gamma \in \Gamma^*_n).
\]
Their paper presents a heuristic argument which suggests that in $\ZZ^d$ the modal geodesics are those which connect the corners of the cube $[0, n]^d$. For $d = 2$, these are the paths $(0, 0) \to (n, 0) \to (n, n)$ and $(0, 0) \to (0, n) \to (n, n)$. Further, one expects that there is some form of monotonicity, roughly saying that paths closer to the corners are more likely to appear as geodesics. This conjecture is borne out by simulation; see Figure \ref{fig:empirical}. This may be compared to the probability of passing through a particular \emph{site}, shown in Figure \ref{fig:siteprobs}.

\begin{figure}
    \centering
    \includegraphics[width=0.7\linewidth]{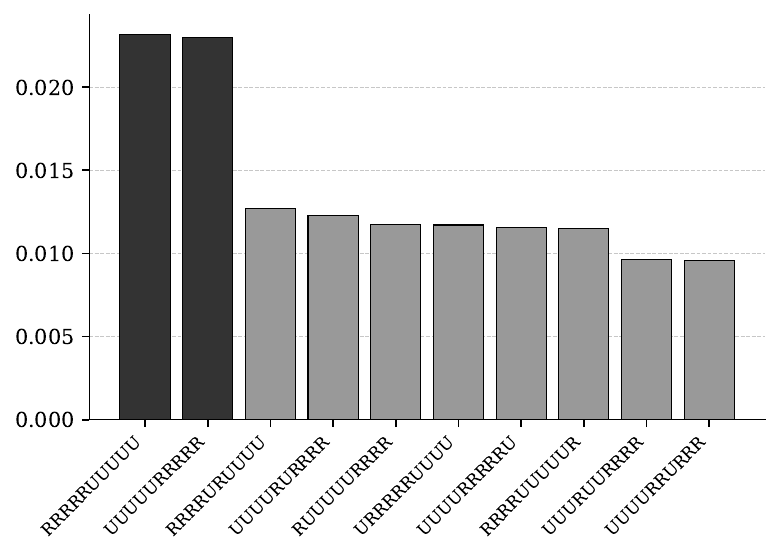}
    \caption{The observed frequency of individual geodesics when $n = 6$, run over $10^6$ simulations with exponential weights. The two corner paths are highlighted and are indeed more probable than any other path by a large margin.}
    \label{fig:empirical}
\end{figure}

\begin{figure}
    \centering
   \includegraphics[width=0.4\linewidth]{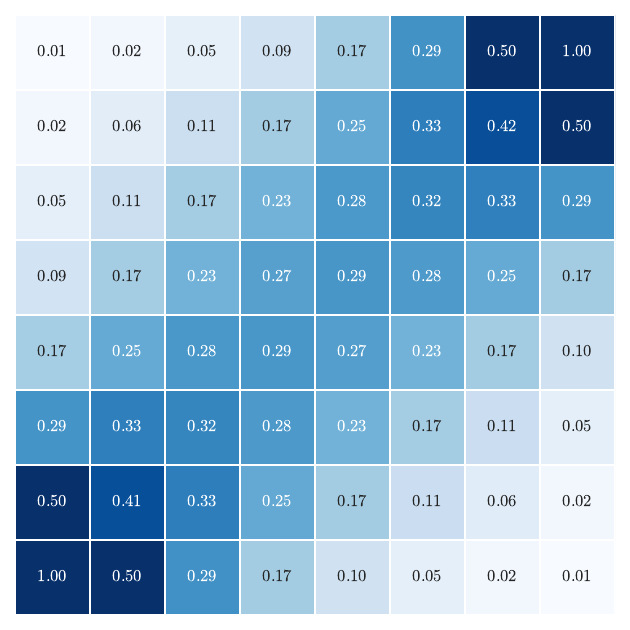}
    \hspace{2em}
   \includegraphics[width=0.4\linewidth]{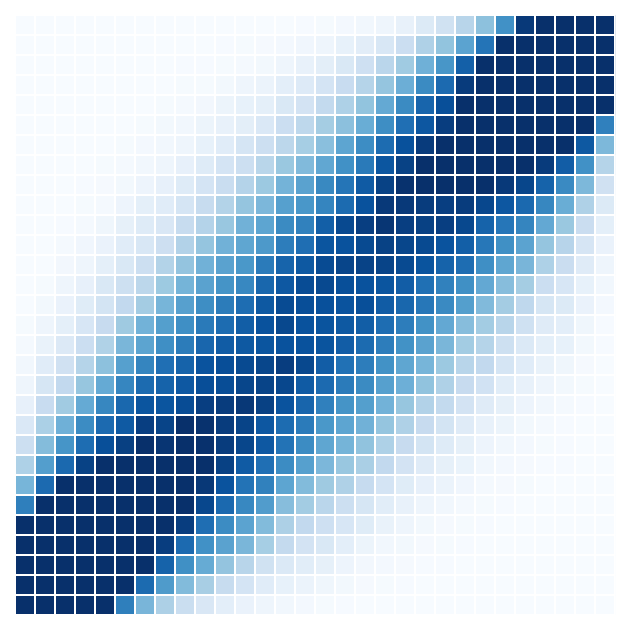}
    \caption{The observed probabilities of a site being visited by a geodesic under exponential weights, when $n = 8$ (left) or $n = 30$ (right). On the level of vertices, those outside of a window of width $O(n^{2/3})$ are exceedingly unlikely to be visited.}
    \label{fig:siteprobs}
\end{figure}

One difficulty is that the appearance of any particular path is an exceedingly rare event. To further hinder us, last passage percolation is thought to belong to the Kardar-Parisi-Zhang (KPZ) universality class, which predicts among other things that the geodesics $\gamma^*_n$ should typically follow the straight line with fluctuations of order $n^{\xi}$, $\xi < 1$. When $d=2$, a value of $\xi = \frac{2}{3}$ has been rigorously established in the exactly solvable case of $\omega_x \sim \Exp(1)$ \cite{BCS:flucts} (and for a related model earlier in \cite{J:transflucts}).

Typically then, one will find $\gamma_n^*$ in the bulk of the cube. However, an overwhelming proportion of directed paths also live in the bulk of the cube. Any path going through the centre will have a very large number of paths which differ by $O(1)$ vertices. To be maximal, a path must in particular beat out these neighbours as it competes for weights. In contrast, a corner path has only $d - 1$ paths which differ by a single vertex, and is the only path which can collect the weights in the corners it visits. There is thus competition between the typical behaviour of wanting to go through the bulk, and increased entropy suppressing the likelihood of seeing a particular path in the bulk.

We prove some results which support the idea that entropy wins and that corner paths are the most likely. We identify three distinct but related mechanisms which advantage corner paths:
\begin{enumerate}[label=(\Alph*)]
    \item As alluded to above, corners in a path introduce vulnerabilities in the form of alternative paths obtained by ``flipping'' the corners. Paths with many corners have their probability of occurring suppressed by this competition.
    \item The initial and final line segments of a path ``feel'' less competition. Paths can afford to collect slightly less weight on these segments. The corner paths maximise the length of these segments and see the most benefit.
    \item Weights near the corners of the cube are accessible by fewer paths. In the planar case, weights \emph{in} the corners are accessible only by a single path each.
\end{enumerate}
We exploit (A) to prove that corner paths are more likely than paths with a number of corners growing linearly in $n$. While not especially intuitive, the effect in (B) falls out of an analysis in the integrable case of exponential LPP. Using the moderate deviation estimates available here, we give two terms in an asymptotic expansion for the maximum probability and show that these are matched by the corner path probability. Finally, we observe (C) in the context of Bernoulli weights with small parameter (decaying with $n$). Here we can prove a rule for comparing geodesic probabilities and conjecture that it holds in general.

\subsubsection*{Note on the use of artificial intelligence}
The authors made use of various large language models in the drafting of this manuscript to check for typographical and mathematical mistakes. Figures and simulations were also produced with the assistance of AI. The mathematical content is due to the authors, with the exception of the estimates proved in Appendix~\ref{app:sharpmd}. An earlier version of this manuscript proved a weaker form of Theorem~\ref{thm:probability} unconditionally, and showed that the stronger result presented here held conditionally on a particular moderate deviation estimate. This estimate was subsequently proved by OpenAI's GPT-5.6 Sol model and incorporated into the text. The proof given in Appendix~\ref{app:sharpmd} has been rewritten by the authors for clarity.

\section{Results}
\subsection*{Notation and preliminaries}
Before stating the results, let us put our notation in order. For vectors $u, v \in \ZZ^d$ with $u \le v$ (coordinate-wise), let $\Gamma(u; v)$ be the set of up-right directed paths connecting $u$ to $v$. To such a path $\gamma \in \Gamma(u; v)$, assign a passage time 
\[
L(\gamma) = \sum_{x \in \gamma} \omega_x.
\]
The last passage time from $u$ to $v$ is
\[
L(u; v) = \max_{\gamma \in \Gamma(u; v)}L(\gamma).
\]
Let $\1 = (1, \dots, 1) \in \ZZ^d$ be the diagonal vector. The $L_n$ we used above is simply $L(0; n\1)$ in this notation. 

When we restrict our attention to $d=2$, we often write coordinates explicitly and use $L(a, b; n, m)$ to denote the passage time from $(a, b)$ to $(n, m)$. In this planar setting, we also make use of point-to-line passage times. If $n \ge a + b$, then define
\[
L^{\bltriangle}(a, b; n) = \max_{0 \le k \le n - (a + b)}L(a, b; a + k, n - a - k),
\]
which will be the greatest passage time taking $(a, b)$ to the antidiagonal at level $n$. The line-to-point passage times $L^{\trtriangle}(n; a, b)$, now with $a + b \ge n$, are defined in the obvious way. Another variety of passage times we consider are half-space passage times $L^{\tltriangle}$, which are point-to-point last passage times except that the admissible paths are those that stay on or above the diagonal connecting the endpoints (that is, where $x_2 - b \ge x_1 - a$). Times $L^{\brtriangle}$ are defined similarly, but restricting to paths on or below the diagonal ($x_1 - a \ge x_2 - b$). It is clear that one has
\begin{equation}
\label{eq:passagetimerelations}
    L^{\tltriangle}(0, 0; n, n) \vee L^{\brtriangle}(0, 0; n, n) \le L(0, 0; n, n) \le L^{\bltriangle}(0, 0; n) + L^{\trtriangle}(n + 1; n, n).
\end{equation}

The time constant $\mu$ defined earlier can be found for arbitrary directions. Let
\[
\mu(x_1, \dots, x_d) = \lim_{n \to \infty}\frac{L(0, \dots, 0; \floor{nx_1}, \dots, \floor{n x_d})}{n}.
\]
That such a limit exists is an application of Kingman's subadditive ergodic theorem. We have abbreviated $\mu = \mu(\frac{1}{d},\dots, \frac{1}{d})$. In $d = 2$ and $\omega_x \sim \Exp(1)$, it is known that $\mu(x, y) = (\sqrt{x} + \sqrt{y})^2$.

Let $\Lambda(t) = \log \EE\sqrbrac{e^{t \omega_x}}$ be the cumulant generating function of the weight distribution. Let $I$ be the large deviation rate function for i.i.d.\ sums of our weights, given by
\[
I(x) = \sup_{t \in \RR} \sqrbrac{tx - \Lambda(t)}. 
\]
If $\gamma \in \Gamma_n$, then Cramér's theorem tells us that for all $\lambda > \EE \omega_x$,
\[
\PP(L(\gamma) \ge d n \lambda) = e^{-d n I(\lambda) + o(n)}.
\]
When $\omega_x \sim \mathrm{Exp}(1)$, we will use a refinement in the form of the Bahadur--Rao theorem \cite{BR:iidsums}. If $S_N$ is a sum of $N$ independent $\mathrm{Exp}(1)$ random variables, then, uniformly for $\lambda$ in compact subsets of $(1,\infty)$,
\begin{equation}
\label{eq:BRthm}
\PP(S_N \ge N\lambda) = \frac{1}{(\lambda - 1)\sqrt{2\pi N}}e^{-N I(\lambda)}(1 + O(N^{-1})).
\end{equation}
With exponential weights we can compute $I(\lambda) = \lambda - 1 - \log \lambda$.

Constants $c$, $C$ will tend to vary from line to line.

\subsubsection*{Paths with many corners are uncompetitive}
Given a path $\gamma \in \Gamma_n$, we say $x \in \gamma$ is a corner if $x - e_{i_1}, \, x + e_{i_2} \in \gamma$ for some $i_1 \ne i_2$. With a fixed $\gamma$ in mind, let $\hat{x} = x$ if $x$ is not a corner, and otherwise $\hat{x} = x - e_{i_1} + e_{i_2}$. In other words, $\hat{x}$ is the result of ``flipping'' the corner at $x$. We will call $\hat{x}$ the \emph{flip} of $x$. Define the set of corners
\[
D(\gamma):=\{x\in \gamma : x\neq \hat{x}\}.
\]
The number of corners of $\gamma$ is also the number of valid paths differing by a single vertex. A path with many corners has many opportunities to be locally suboptimal. If our heuristic is correct, we should expect that the probability decreases as $\abs{D(\gamma)}$ increases. 

Our first result validates this in the case that the number of corners is linear in $n$. Here and in what follows, write $\cornerpath$ for the top-left corner path $(0,\ldots, 0) \to (0,\ldots,0 ,n)\to (0,\ldots,0,n ,n)\to \dots \to (n,n,\ldots,n, n)$. Observe that $\abs{D(\cornerpath)} = d - 1$ and that this is minimal.

We require two assumptions on the weight distribution, one technical and one only partly so. The first is superexponential decay on the lower-tail of the weights, to avoid unexpectedly small passage times.
\begin{assumption}
\label{ass:lower-tail}
    There exist $\nu > 1$ and $c > 0$ such that     for all $t > 0$, 
    \begin{align*}
        \PP(\omega_x< -t) < e^{-c t^\nu}.
    \end{align*}

\end{assumption}
Note that this assumption is satisfied automatically if $\omega_x \ge 0$ almost surely. We need this assumption for a lower-tail large deviation bound proved in Appendix~\ref{app:lowertailldp}.

For the second, recall that $\Lambda(t) = \log \EE\sqrbrac{e^{t \omega_x}}$ is the cumulant generating function of the weight distribution and assume that it is finite in some interval about the origin. 
\begin{assumption}
\label{ass:tiltable}
    Suppose that the time constant $\mu$ is \emph{tiltable}, in the sense that $\mu < \sup_{t \ge 0} \Lambda'(t)$. 
\end{assumption}

We comment on the meaning of this assumption at the end of Section~\ref{sec:proofcorners}. It rules out the critical and supercritical regimes, but is provably satisfied by most sensible distributions in the subcritical regime. In particular, bounded distributions without a large atom on their maximum, or unbounded distributions with exponential moments all finite, both satisfy the assumption.

\begin{theorem}
\label{thm:corners}
Suppose the weight distribution satisfies Assumptions~\ref{ass:lower-tail} and \ref{ass:tiltable}, and fix $\epsilon>0$. There is $C_\epsilon > 0$ such that for all sufficiently large $n$, any path $\gamma \in \Gamma_n$ with $\abs{D(\gamma)} > \epsilon n$ has
\[
\frac{\PP(\gamma \in \Gamma_n^*)}{\PP(\cornerpath\in \Gamma_n^*)} \le C_\epsilon e^{- n / C_\epsilon}.
\]
\end{theorem}
We make use of a general lower bound on $\PP(\cornerpath\in \Gamma_n^*)$ first established for $d = 2$ in \cite{ABGS:midpoint}. We prove that this bound holds for all $d$, under the simplifying assumption of tiltability, in Appendix~\ref{app:cornerpathprob}.

\begin{corollary}
\label{cor:corners}
Suppose the weight distribution satisfies Assumptions~\ref{ass:lower-tail} and \ref{ass:tiltable}. Write
\[
\Gamma^\ge_n = \set[\big]{\gamma \in \Gamma_n : \PP(\gamma \in \Gamma_n^*) \ge \PP(\cornerpath \in \Gamma_n^*)}
\]
for the set of paths which are at least as likely to appear as geodesics as $\cornerpath$. For any $\epsilon > 0$, it holds for all $n$ large enough that
\[
\abs{\Gamma^\ge_n}\leq e^{\epsilon n} \quad \text{and} \quad \frac{\abs{\Gamma^\ge_n}}{\abs{\Gamma_n}} \le d^{- d n (1 - \epsilon)}.
\]
\end{corollary}
Observe that neither inequality holds with $\epsilon = 0$.

\subsubsection*{The corner path probability in exponential LPP}
From now on we work in $d = 2$. We consider the integrable case of exponential LPP and give an asymptotic upper bound for the probability of the modal geodesic. We compute two terms in an asymptotic expansion for the probability of seeing $\cornerpath$.

Although we use various other inputs for the integrable model, we use most crucially the moderate deviation estimates available here. With $T > 0$ a large constant and $\rho > 0$ a small constant, the following holds uniformly for $T \le t \le \rho n^{2/3}$:
\begin{equation}
\label{eq:lefttail}
    \log \PP(L_n \le 2 \mu n -  t n^{1/3}) = -\frac{1} {192} t^{3} + O(t^4 n^{-2/3}) + O(\log t).
\end{equation}
This asymptotic was proved in \cite{BDMMZ:tail} for geometric weights, but their steepest descent calculation may be carried out for exponential weights also. See \cite{L:iteratedlog} for an indication of the proof in the exponential case.

It will be important to know too the lower-tail moderate deviations for point-to-line times. A completely optimal expansion comparable to \eqref{eq:lefttail} does not appear in the literature, but bounds of the correct order have been established in \cite{LR:betadeviations,BGHK:lpp-lis}, and the correct leading constants have been found (in a restricted range of $t$) in \cite[Theorem 1.6]{BBBK:tail-bounds}. We present in Appendix~\ref{app:sharpmd} an argument which transfers the point-to-point estimate \eqref{eq:lefttail} to the point-to-line and half-space settings. The available literature treats estimates for general $\beta$-ensembles, while the argument we give uses in an essential way certain exact relations between the ``classical'' $\beta$-ensembles ($\beta = 1,\, 2,\, 4$) corresponding to our passage times. In this way, we do not improve on existing results in their full generality.

Write $\gamma^*_n$ for the almost surely unique geodesic under these continuous weights.
\begin{theorem}
\label{thm:probability}
Suppose $d = 2$ and  $\omega_x \sim \mathrm{Exp}(1)$, and let $I(\lambda) = \lambda - 1 - \log \lambda$ be the large deviation rate function for the sum of i.i.d.\ exponential random variables.
\begin{multline*}
\lim_{n \to \infty}
\frac{\log \PP (\cornerpath = \gamma^*_n) + 2n I(\mu)}
     {\sqrt{2 n}} \\
= \lim_{n \to \infty}
\frac{\max_{\gamma \in \Gamma_n} \log \PP (\gamma = \gamma^*_n)
      + 2n I(\mu)}
     {\sqrt{2 n}}
= \frac{4}{3}.
\end{multline*}
\end{theorem}
The expression $\PP (\cornerpath = \gamma^*_n) \approx e^{-2n I(\mu)}$, without the $O(\sqrt{n})$ correction, has already been established by \cite{ABGS:midpoint}. For exponential weights the value $\mu = 2$ is known, and here $I(\mu) = 1 - \log 2$. Thus we may write
\begin{equation}
    \label{eq:cornerprob}
        \PP (\cornerpath = \gamma^*_n) \approx  \brac[\Big]{\frac{2}{e}}^{2n - \kappa \sqrt{2n} + o(\sqrt n)} \approx \max_{\gamma \in \Gamma_n} \PP (\gamma = \gamma^*_n),
\end{equation}
where $\kappa = \frac{4}{3(1 - \log 2)}$. To avoid overstatement, note that the ratio of the two probabilities may still diverge as $e^{o(\sqrt{n})}$.

With the relatively precise information from Theorem \ref{thm:probability}, we can extend Theorem \ref{thm:corners} to the case of paths with corners growing like $n^{1/2 + \epsilon}$. We can also compare probabilities for certain paths with a constant number of corners. For example, let $\gamma^k_n$ be the \emph{staircase} path with $2k$ segments,
\[
\gamma^k_n = (0, 0) \to (0, \floor{n / k}) \to (\floor{n / k}, \floor{n / k}) \to \cdots \to (n, n).
\]
Here the segments are alternately vertical and horizontal of length $\floor{n / k}$ or $\ceil{n / k}$. Note that $\abs{D(\gamma^k_n)} = 2k - 1$ does not grow with $n$.

\begin{corollary}
\label{cor:staircases}
Assume $d=2$. Suppose $\omega_x \sim \mathrm{Exp}(1)$. For all $k \ge 5$ and all $n$ large enough,
    \[
        \PP (\cornerpath = \gamma^*_n) > \PP (\gamma^k_n = \gamma^*_n).
    \]
\end{corollary}

To prove Corollary \ref{cor:staircases}, we find an upper bound on the probability of seeing a staircase path which is asymptotically smaller than the one given by Theorem~\ref{thm:probability}. The set of paths we consider could also be expanded somewhat, but our crude method of producing a bound does not admit a significant degree of generality. Bringing the lower bound to $k \ge 5$ makes full use of the sharp point-to-line and half-space estimates in Appendix~\ref{app:sharpmd}. Establishing the inequality for all $k \ge 2$ would seem to require a more refined technique.

\subsubsection*{Monotonicity of path probabilities}
Lastly, we look at weights $\omega_x \sim \mathrm{Ber}(p)$ for small values of $p$. In this setting, questions of path probabilities boil down to certain tractable counting problems, and we may give the complete picture. 

Path probabilities are expected to be monotonic in that paths closer to the corners (in some sense) ought to occur more often. We give an ordering which we conjecture captures this monotonicity. 

It will be easiest to define the ordering in terms of the following encoding of paths, whereby we describe $\gamma \in \Gamma_n$ as a sequence $(a_1, a_2, \dots, a_r)$. The signs of the entries alternate and the sums of the positive and negative parts are both $n$. A positive value of $a_1$ represents the length of an initial vertical segment of $\gamma$ starting from the origin, with a negative value representing the length of an initial horizontal segment. Then $a_2$ is the length of the subsequent segment (with sign indicating the direction), and so on. For example, $\cornerpath$ would be encoded as $(n, -n)$, and a staircase path $\gamma^k_n$ would be $(n / k, - n / k, \dots, n / k, - n / k)$, when $k$ divides $n$. Observe that the length of the sequence $r$ is exactly $\abs{D(\gamma)} + 1$. 

Define also the partial sums $s_k = \sum_{i = 1}^{k}a_i$, which give the distance of the path from the diagonal at the $k$-th corner, and set $s_0 = a_0 =0$. We consider two moves transforming a path $\gamma$.

\begin{itemize}
    \item A \emph{bump} move, smoothing out a concavity lying on one side of the diagonal. If $s_k \ge 0$, $a_{k} < 0$ and $k \le r - 2$, then we replace 
    \begin{multline*}
        (a_1, \dots, a_k, \dots, a_r) \\ \mapsto (a_1, \dots, a_{k - 2}, a_{k - 1} + a_{k + 1}, a_{k} + a_{k + 2}, a_{k + 3}, \dots, a_{r}).
    \end{multline*}
    This has the effect of taking a south-east-facing concavity above the diagonal and filling it in, removing two corners in the process. The move can also be performed under the symmetric situation below the diagonal, when $s_k \le 0$, $a_{k} > 0$ and $k \le r - 2$.

    \item An \emph{unwinding} move, reducing the number of crossings of the path across the diagonal. If $l < k$ are such that $s_l,\, s_k > 0$ but $s_i < 0$ for $l + 1 \le i \le k - 1$, then we replace
    \begin{multline*}
        (a_1, \dots, a_l, \dots, a_k, \dots, a_r) \\\mapsto (a_1, \dots, a_{l} + a_{k}, a_{k - 1}, \dots, a_{l + 2}, a_{k + 1} + a_{l + 1}, a_{k + 2}, \dots, a_{r}).
    \end{multline*}
    Here we take the corners just before and just after a path dips below the diagonal, and reflect it around the line joining those corners. The same move can be performed when the path peaks above the diagonal, where $s_l,\, s_k < 0$ but $s_i > 0$ for $l + 1 \le i \le k - 1$. As a boundary case, we also allow $l = 0$, with $s_0 = 0$, where we interpret the displayed replacement with $a_0=0$ and omitting nonexistent terms before $a_1$.
\end{itemize}

\begin{figure}
    \centering
    \resizebox{\textwidth}{!}{\input{fig_bump}}
    \caption{Bumping a corner of a path. The concavity formed by $a_4$ and $a_5$ does not lie entirely above the diagonal and may not be bumped.}
    \label{fig:pathbump}
\end{figure}
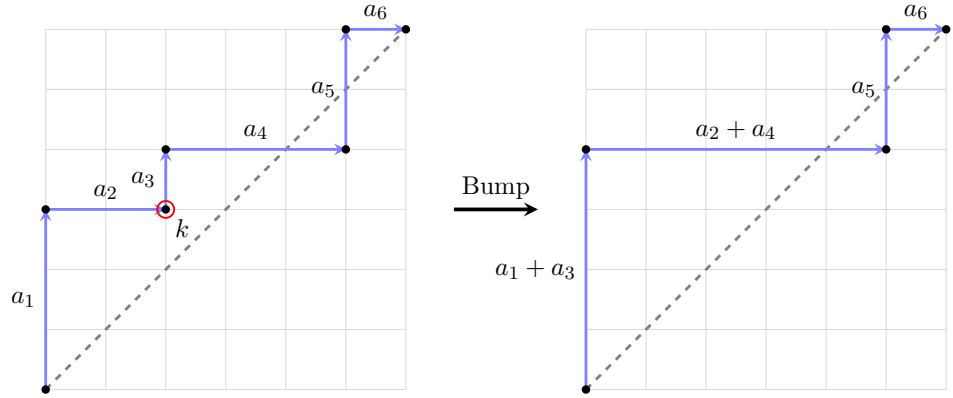

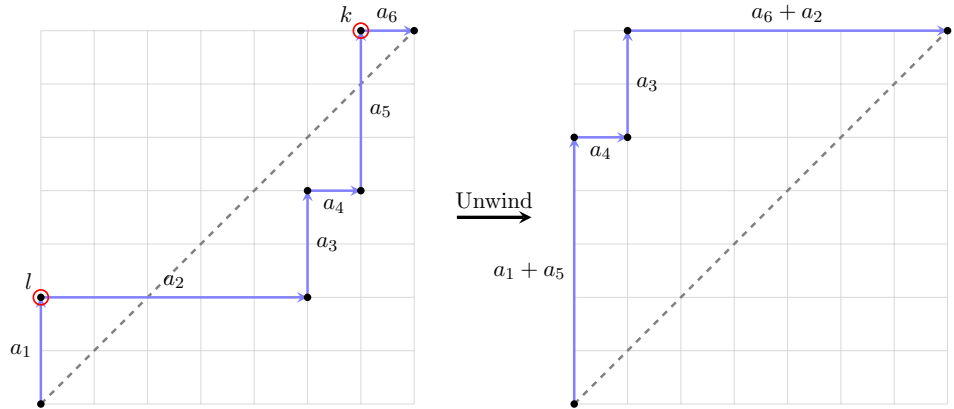
\begin{figure}
    \centering
    \resizebox{\textwidth}{!}{\input{fig_unwind}}
    \caption{Unwinding a path about two corners. Observe that the order of edges is reversed in the modified segment.}
    \label{fig:pathunwind}
\end{figure}

Figures \ref{fig:pathbump} and \ref{fig:pathunwind} illustrate the operations. We expect that these operations will strictly increase the probability that a path appears as a geodesic. We can prove this in the small $p$ regime. Write $\pi \succ \gamma$ if $\pi$ may be obtained from $\gamma$ by performing these moves repeatedly. 

\begin{theorem}
\label{thm:smallpmono}
Assume $d=2$. There is $p_{n} > 0$ such that if $\omega_x \sim \mathrm{Ber}(p)$ with $0 < p \le p_n$, we have for all paths $\pi,\, \gamma \in \Gamma_n$ the implication
    \begin{equation}
        \label{eq:MostLikelyPath}
        \pi \succ \gamma \Rightarrow \PP(\pi \in \Gamma^*_n ) \ge \PP(\gamma \in \Gamma^*_n).
    \end{equation}
    Moreover, this inequality is strict unless $\pi = \gamma$.
\end{theorem}

We do not obtain a uniform lower bound for $p_n$, and indeed the lower bound furnished by our proof decays polynomially in $n$.

Observe that $\cornerpath$ and its reflection $\rcornerpath$ are the only maximal elements under the $\succ$ ordering and at least one of these can be obtained beginning from a given path. To see this, consider unwinding until the path lies on one side of the diagonal, and then bumping the corners until it is a corner path. The theorem then implies that the probabilities are also maximal.
\begin{corollary}
    \label{cor:smallpmode}
    There is $p_{n} > 0$ such that if $\omega_x \sim \mathrm{Ber}(p)$ with $0 < p \le p_n$, then for any path $\gamma$,
    \begin{equation*}
        \PP(\cornerpath \in \Gamma^*_n ) \ge \PP(\gamma \in \Gamma^*_n).
    \end{equation*}
    This inequality is strict unless $\gamma \in \set{\cornerpath, \rcornerpath}$.
\end{corollary}

\begin{remark}
    This $\succ$ is not a total ordering on $\Gamma_n$. One can of course define total orderings on the set of paths, but there is no one ordering with respect to which path probabilities are monotonic, in general. For fixed $n$, one can use software to find pairs of paths whose probabilities are larger or smaller than the other, depending on the weight distribution.
\end{remark}

\section{Proof of Theorem~\ref{thm:corners}}
\label{sec:proofcorners}
We proceed by relating the probability of observing a corner geodesic with that of seeing a large passage time along a fixed path, following \cite{ABGS:midpoint}, and then observing an exponential loss resulting from the competition near corners in our path. 

\begin{proof}[Proof of Theorem~\ref{thm:corners}]
It is shown in \cite{ABGS:midpoint} when $d = 2$ that for every $\delta > 0$ and all sufficiently large $n$, one has
\begin{equation}\label{eq:lower-gamma1}
\PP(\cornerpath\in \Gamma_n^*) \ge \exp\bigl(-d n (I(\mu)+\delta)\bigr).
\end{equation}
We give a proof for general $d \ge 2$ in Appendix~\ref{app:cornerpathprob}.

This handles the lower bound on the denominator. For an upper bound on the numerator, our strategy will be as follows. For a path to be a geodesic, it must at least have a passage time near $d n \mu$. Conditioned on such a large passage time, the weights on the path look like i.i.d.\ exponential tilts of the original distribution. Further, the weights on the corners must at least be as big as the weights on the flipped vertices, as otherwise we could find a new path with a larger passage time. The probability that the corner weight is bigger than its flip is strictly less than $1$, and imposing this condition on the linear number of corners incurs an exponential cost.

Take another small constant $\alpha > 0$, to be chosen later. A union bound gives
\begin{equation}
\label{eq:cornerunionbound}
\PP(\gamma \in \Gamma^*_n) \le \PP\bigl(\gamma \in \Gamma^*_n, L_n \ge dn(\mu - \alpha)\bigr) + \PP\bigl(L_n < dn(\mu - \alpha)\bigr). 
\end{equation}
For the second probability, we use the bounds of \cite{CZ:fpp-ldp}, established for first passage percolation but valid also in LPP (see Lemma~\ref{lemma:lower-tail} for a self-contained proof): 
\[
    \PP\bigl(L_n < dn(\mu - \alpha)\bigr) \le \exp(-dn (I(\mu)+1))
\]
 This is negligible in our calculation.

We turn to the first probability in \eqref{eq:cornerunionbound}. This is itself bounded by
\[
    \PP\bigl(\gamma \in \Gamma^*_n, L_n \ge dn(\mu - \alpha)\bigr) \le \PP\bigl(L(\gamma) \ge dn(\mu - \alpha),\, \omega_x \ge \omega_{\hat{x}} \text{ for all } x \in D(\gamma)\bigr),
\]
where recall that $\hat{x}$ is the flip of $x$.
Consider a tilted measure for the weights on $\gamma$ with Radon-Nikodym derivative $e^{\theta_\alpha x - \Lambda(\theta_\alpha)}$, where $\theta_\alpha > 0$ is chosen so that the new measure has mean $\mu - \alpha$ on $\gamma$. This is possible by the assumption of tiltability. Observe that the exponent evaluated at $x = \mu - \alpha$ is exactly $I(\mu - \alpha)$. Write $\EE_\alpha$ for the expectation under this measure. Then
\begin{align*}
   & \PP\bigl(\gamma \in \Gamma^*_n, L_n \ge dn(\mu - \alpha)\bigr) \\
   &\le \EE\sqrbrac[\bigg]{\bone_{L(\gamma) \ge dn(\mu - \alpha)} \prod_{x \in D(\gamma)} \bone_{\omega_x \ge \omega_{\hat{x}}}}\\
    &= \EE_\alpha \sqrbrac[\bigg]{e^{-\theta_\alpha L(\gamma) + (d n + 1)\Lambda(\theta_\alpha)}\bone_{L(\gamma) \ge dn(\mu - \alpha)} \prod_{x \in D(\gamma)} \bone_{\omega_x \ge \omega_{\hat{x}}}}\\
    &\le \EE_\alpha \sqrbrac[\bigg]{e^{-\theta_\alpha \, dn(\mu - \alpha) + (d n + 1)\Lambda(\theta_\alpha)}\bone_{L(\gamma) \ge dn(\mu - \alpha)} \prod_{x \in D(\gamma)} \bone_{\omega_x \ge \omega_{\hat{x}}}}.
    \end{align*}
Since $\theta_\alpha (\mu - \alpha) -\Lambda(\theta_\alpha) = I(\mu-\alpha)\geq 0$, this is bounded from above by
    \begin{align*}
& e^{-dn I(\mu - \alpha)+\Lambda(\theta_\alpha)} \EE_\alpha \sqrbrac[\bigg]{ \prod_{x \in D(\gamma)} \bone_{\omega_x \ge \omega_{\hat{x}}}}\notag\\
    &\quad = e^{\Lambda(\theta_\alpha)}e^{-dn I(\mu - \alpha)} \PP(\omega_\alpha \ge \omega')^{\abs{D(\gamma)}}.
\end{align*}
In the last line, $\omega_\alpha$ is a weight drawn under the tilted measure and $\omega'$ is an independent weight drawn from the original distribution. Write $p_\alpha = \PP(\omega_\alpha \ge \omega')$. Since these two variables have the same range and are independent and non-constant, we see that $p_\alpha < 1$. Moreover, decreasing $\alpha$ only makes $\omega_\alpha$ larger and increases $p_\alpha$. They are then uniformly bounded by $p:= p_0  < 1$. It is at this point that the proof would break down if we were to allow supercriticality, as in this case the conditioned weights on the corner path would need to be constant, and we would have $p_0 = 1$.

Using $\abs{D(\gamma)} \ge \epsilon n$, we get finally the bound
\begin{equation}\label{eq:upper-gamma1}
    \PP\bigl(\gamma \in \Gamma^*_n, L_n \ge dn(\mu - \alpha)\bigr) \le e^{\Lambda(\theta_\alpha)}\exp\bigl(-n(dI(\mu - \alpha) + \epsilon \log p^{-1})\bigr).
\end{equation}
Now choose $\alpha$ small enough to ensure that 
\[
d I(\mu - \alpha) + \epsilon \log p^{-1} \ge d I(\mu) + 2 c_\epsilon,
\]
for some $c_\epsilon \in (0,1/2)$ small.

 Choose $\delta < \min\{1/2,c_\epsilon/d\}$ and $C:=2e^{\Lambda(\theta_\alpha)}$. Combining \eqref{eq:lower-gamma1} with \eqref{eq:upper-gamma1}, we have 
\[
\frac{\PP\bigl(\gamma \in \Gamma^*_n)}{\PP(\cornerpath\in \Gamma_n^*)} \le e^{\Lambda(\theta_\alpha)} e^{- n(d I(\mu) + 2 c_\epsilon - d I(\mu) - d \delta)} + e^{-dn/2} \le C e^{- n c_\epsilon}.
\]
\end{proof}

\begin{proof}[Proof of Corollary~\ref{cor:corners}]
This is essentially an exercise in enumeration. Fix the $\epsilon>0$ appearing in the claim  and set
\[
    \delta=\frac12\min\{\epsilon,d\epsilon\log d\}.
\]
Choose $\eta>0$ sufficiently small that
\[
    \eta\log\!\left(\frac{2ed^2}{\eta}\right)<\frac{\delta}{2}.
\]
By Theorem~\ref{thm:corners}, for $n$ sufficiently large, every
$\gamma \in \Gamma_n$ with $\abs{D(\gamma)} > \eta n$ satisfies
$\PP(\gamma \in \Gamma_n^*) < \PP(\cornerpath \in \Gamma_n^*)$. Hence
$$\Gamma_n^\ge \subseteq \set{\gamma \in \Gamma_n : \abs{D(\gamma)} \le
\eta n}.$$

A path in $\Gamma_n$ is a word of length $dn$ containing $n$ copies of each
coordinate direction. If a path consists of \(r\) maximal straight segments, then it has \(r-1\) corners.
   Thus $\abs{D(\gamma)} \le \eta n$ implies $r\le \eta n+1$.

Fix $r\le \eta n+1$. We obtain a crude upper bound as follows:
\begin{itemize}
\item Choose $r-1$ corner locations among the $dn$ possible
positions, yielding $\binom{dn}{r-1}$ choices.
\item Choose the direction of each run. There are $d$ choices for the first run and at most $d-1$
choices for each subsequent run, hence at most
$d(d-1)^{r-1}\le d^r$ possibilities.
\end{itemize}

Ignoring the constraint that each direction must appear exactly $n$ times
(making this an upper bound), we get
\[
\abs{\{\gamma\in\Gamma_n:\text{$\gamma$ has $\le \eta n$ corners}\}}
\le
\sum_{r\le \eta n+1} d^r \binom{dn}{r-1}.
\]

For $r\le \eta n+1$ with $n$ large enough, we have
\[
d^r \binom{dn}{r-1}
\le
\left(\frac{2e d^2 n}{r}\right)^r
\le
\left(\frac{2e d^2}{\eta}\right)^{\eta n+1}.
\]

Taking logarithms and dividing by $n$,
\[
\frac{1}{n} \log\abs{\{\gamma\in\Gamma_n:\text{$\gamma$ has $\le \eta n$ corners}\}}
\le
\eta \log\!\left(\frac{2 e d^2}{\eta}\right)
+ o(1).
\]
Consequently, by the choice of $\eta$, for all sufficiently large $n$,
\[
    \abs{\Gamma_n^\ge}\le e^{\delta n}\le e^{\epsilon n}.
\]
On the other hand,  a standard combinatorial argument shows 
\[
\frac{1}{n}\log \abs{\Gamma_n} = d\log d + o\!\left(1\right).
\]
Thus, we conclude 
\[
\frac{\abs{\Gamma_n^\ge}}{\abs{\Gamma_n}} \le
\exp\{-dn\log d+2\delta n\}
\le d^{-dn(1-\epsilon)}.
\]
\end{proof}

\begin{remark}
    The assumption of tiltability ensures that there is an exponential tilt of the weight distribution having mean $\mu$. It was necessary for the proof and serves two purposes. Most significantly, it rules out the critical and supercritical cases. If $M = \esssup \omega_x$ and $\PP (\omega_x = M) \ge p_c$, the critical probability for directed site percolation, then we see that $\mu = M$. Unless the weight distribution is constant, no exponential tilt can attain $M$ as the mean. This is an if and only if statement: a distribution bounded above will fail to have tiltability if and only if the atom at the maximum is at or above the critical probability.

    We should not expect an exponential bound of the type in Theorem~\ref{thm:corners} in the critical and supercritical cases. At or above criticality, all but a vanishing proportion of the weights along a high-weight path are constant and equal to the maximum of the distribution. There is little scope to improve the passage time by flipping corners, and as such the presence of corners does not incur an exponential cost as in the subcritical case.
\end{remark}
\begin{remark}
    The second purpose of the tiltability assumption is to remove certain pathological examples of distributions which have finite exponential moment but nevertheless suffer somewhat heavy tails. An unbounded distribution having $R = \sup \Lambda'(t) < \infty$ is sometimes called \emph{non-steep}. A representative example is the variable on $[1, \infty)$ with density proportional to $x^{-3}e^{-x}$, which one can verify is non-steep.
    
    When $\mu > R$, the passage times are dominated by a small number of atypically large weights. This change in qualitative behaviour is an example of the big-jump phenomenon that appears when an i.i.d.\ sum is conditioned to have a mean larger than $R$. See \cite{DDS:bigjumps} for some results in this direction. Our proof cannot be carried out as written -- one would have to distinguish between the typical and the atypically large values on a path and adjust accordingly. However, the heuristic of the argument remains valid, and we suspect the proof can be modified to encompass all subcritical distributions with a finite exponential moment.

    We note also that there exist distributions for which $\mu > R$. Take for example the density  proportional to $x^{-3}e^{-x}$ on $[1, \infty)$, for which $R < \infty$. The time constant is trivially bounded below by the mean of the maximum of $d$ independent copies of $\omega_x$. This quantity diverges as $d$ increases, and in doing so we may make it larger than $R$.
\end{remark}

We end with two short lemmas providing criteria for tiltability.

\begin{lemma}
    Suppose $\esssup \omega_x = \infty$ and there is $\delta > 0$ such that $\EE\sqrbrac{e^{\delta \omega_x}} < \infty$. Let $R = \sup_{t\geq 0} \Lambda'(t)$. Sufficient conditions for a weight distribution to be tiltable are $R = \infty$ or $I(R) > \log d$. In particular, if $\EE\sqrbrac{e^{t \omega_x}} < \infty$ for all $t > 0$, then the distribution is tiltable.
    \end{lemma}

    \begin{proof}
        Recall that our goal is to show $\mu < R$. Having finite exponential moment is enough to ensure $\mu < \infty$, so if $R = \infty$ there is nothing to prove. So suppose $R < \infty$ and $I(R) > \log d$. By continuity, we can find $\eta > 0$ such that $I(R - \eta) > \log d$. For each $\gamma \in \Gamma_n$, we have $\PP(L(\gamma) > d n (R - \eta)) = e^{-(1+o(1)) d n I(R - \eta)}$. Also, we saw in the proof of Corollary~\ref{cor:corners} that $\log \abs{\Gamma_n} =(1+o(1)) d n \log d$. A union bound then gives that
        \[
        \PP(L_n > d n (R - \eta)) \le e^{d n \log d - d n (1 + o(1)) I(R - \eta)} = C_\epsilon e^{dn (\log d - I(R - \eta) + o(1))}.
        \]
The exponent is negative and this probability vanishes, showing that $\mu \le R - \eta < R$.

To prove the last statement, suppose now that $\EE[e^{t\omega_x}]<\infty$ for every $t$. We prove that $\Lambda'(t)\to \infty$ as $t\to\infty$. Recall that
\[
\Lambda'(t)=
\frac{\EE[\omega_x e^{t\omega_x}]}{\EE[e^{t\omega_x}]}.
\]
Fix \(R_0>0\). Since $\omega_x\ge R_0-(R_0-\omega_x)_+$, 
we have
\[
\Lambda'(t) \ge
R_0-
\frac{\EE[(R_0-\omega_x)_+e^{t\omega_x}]}{\EE[e^{t\omega_x}]}.
\]
Some calculus shows that
\[
\EE[(R_0-\omega_x)_+e^{t\omega_x}] \le \frac{e^{tR_0-1}}{t}.
\]
On the other hand,
\[
\EE[e^{t\omega_x}] \ge e^{t(R_0+1)} \PP(\omega_x>R_0+1).
\]
Since $\omega_x$ is unbounded above,
\(
\PP(\omega_x>R_0+1)>0.
\) 
Consequently,
\[
\frac{\EE[(R_0-\omega_x)_+e^{t\omega_x}]}{\EE[e^{t\omega_x}]}
\le \frac{e^{-t-1}} {t \PP(\omega_x>R_0+1)}
\longrightarrow 0.
\]
Thus,
\[
\liminf_{t\to\infty}\Lambda'(t)\ge R_0.
\]
Since $R_0>0$ is arbitrary, it follows that
$\Lambda'(t) \to \infty$ and hence that $R = \infty$.
\end{proof}

\begin{lemma}
    Suppose $M = \esssup \omega_x < \infty$. Then the distribution is tiltable if and only if $\PP(\omega_x = M) < p_c$, the critical probability for directed site percolation.
\end{lemma}
\begin{proof}
        The assumption of subcriticality is equivalent to $\mu < M$ and implies in particular that the distribution is nonconstant. (The proof may be adapted from \cite{M:strictfpp}.) We also have $\mu > \EE \sqrbrac{\omega_x}$. Thus $\mu$ lies strictly between the mean and the maximum value of $\omega_x$. It is a standard fact about exponential tilts that a bounded variable may be tilted to have any mean in this open interval.

        On the other hand, if $\PP(\omega_x = M) \ge p_c$, then $\mu = M$, while the mean under every finite exponential tilt is strictly smaller than $M$ unless the distribution is constant. In the constant case, the strict inequality for tiltability also fails.
    \end{proof}

\section{Proof of Theorem~\ref{thm:probability}}
\label{sec:proofprobability}
The upper bound can be shown quite directly.
\begin{proof}[Proof of the upper bound in Theorem~\ref{thm:probability}]
Take an arbitrary $\gamma \in \Gamma_n$. We can write
\[
    \PP(\gamma = \gamma^*_n) = \PP(L(\gamma) = L_n).
\]
The last event is contained in $\bigcup_{a \in \ZZ} \set{L(\gamma) \ge a, L_n \le a + 1}$, and so we get a union bound 
\[
    \PP(\gamma = \gamma^*_n) \le \sum_{a = 0}^\infty \PP(L(\gamma) \ge a, L_n \le a + 1).
\]
Now observe that $\set{L(\gamma) \ge a}$ is an increasing event in the weights, while $\set{L_n \le a + 1}$ is decreasing. The Harris inequality will then imply that they are negatively correlated. Additionally, the bound is uniform over $\gamma$. Choosing some arbitrary $\gamma_0$ as a placeholder, we end up with
\begin{equation}
\label{eq:maxunionbound}
        \max_{\gamma \in \Gamma_n}\PP(\gamma = \gamma^*_n) \le \sum_{a = 0}^\infty \PP(L(\gamma_0) \ge a)\PP(L_n\le a + 1).
\end{equation}
We can estimate these probabilities with \eqref{eq:BRthm} and \eqref{eq:lefttail}. Introduce a new index $t=t(a)$ with $a = 2 \mu n - t n^{1/3}$. When $t \ll n^{2/3}$, set
\[
    N=2n+1,\qquad \lambda_{n,t}=\frac{2\mu n-t n^{1/3}}{N}.
\]
The rate function for $L(\gamma)$ and $\mu=2$ gives
\begin{align}
    \PP(L(\gamma_0) \ge 2 \mu n - t n^{1/3})
    &= \frac{e^{-N I(\lambda_{n,t})}}{(\lambda_{n,t}-1)\sqrt{2\pi N}}(1+O(N^{-1}))\label{eq:lowertailExpldp}\\
    &= \exp\brac[\big]{-2 n I(\mu) + \frac{1}{2}t n^{1/3}
    +O(t^2n^{-1/3})+O(\log n)}.\notag
\end{align}
On the other hand, we have the moderate deviations of the passage time when $T \le t \le \rho n^{2 / 3}$, where $T,\rho>0$ are the large and small constants, respectively, as specified above \eqref{eq:lefttail}. In this range we have
\begin{equation}\label{eq:lefttail_revisit}
    \PP(L_n \le 2 \mu n - t n^{1/3} + 1) = \exp\brac[\big]{-\frac{1}{192}t^{3} + O(t^4 n^{-2/3}) + O(\log t)}.
\end{equation}

 For $t\leq T$, we control the first factor $\PP(L(\gamma_0) \ge a)$. We use the exact tail identity
    \[
        \PP(S_N\ge a)=e^{-a}\sum_{j=0}^{N-1}\frac{a^j}{j!},
    \]
where \(S_N\) is the sum of \(N\) independent rate one exponential random variables and has the same distribution as \(L(\gamma_0)\). For $a \ge 4n - T n^{1/3}\gg 1$, the sum is bounded by a constant multiple of its final term. Stirling's formula therefore gives
\[
        \log\PP(S_{2n+1}\ge a)=-(a-2n)+2n\log(a/(2n))+O(\log n).
    \]
    Recall that $\mu = 2$ and $I(x)= x-1 -\log(x)$ is the rate function for rate one exponential random variables. In particular, $I(\mu)= I(2)= 1-\log 2$. The total contribution from such $a$ is hence
    \begin{equation}
        \label{eq:tailsum}
        \begin{aligned}
            \sum_{a = \floor{4n - T n^{1 / 3}}}^\infty \PP(&L(\gamma_0) \ge a)\PP(L_n \le a + 1)\\
            &\le \sum_{a = \floor{4n - T n^{1 / 3}}}^\infty \exp\brac[\big]{-(a - 2n) + 2 n \log(a / 2 n) + O(\log n)}\\
            &\le  \exp\brac[\big]{ - 2 n (1 - \log 2) + T n^{1 / 3} + O(\log n)}\\
            &= \exp\brac[\big]{ - 2 n I(\mu) + T n^{1 / 3} + O(\log n)}.
        \end{aligned}
    \end{equation}  

The terms coming from $t \ge \rho n^{1 / 2}$ can be controlled through the second factor. The corresponding range for $a$ is $0 \le a \le 2 n \mu - \rho n^{5/6}$, and here we have
    \begin{align*}
        \PP(L(\gamma_0) \ge a)\PP(L_n \le a + 1) &\le \PP(L_n \le a + 1)\\
        &\le \PP(L_n \le 2 n \mu - \rho n^{5/6} + 1)\\
        &\le \exp(-c_\rho n^{3/2}),
    \end{align*}
    where the last line follows from \eqref{eq:lefttail}. Thus
    \begin{equation}
        \label{eq:headsum}
        \begin{aligned}    
            \sum_{a = 0}^{\ceil{2 n \mu - \rho n^{5/6}}} \PP(L(\gamma_0) \ge a)&\PP(L_n \le a + 1)\\
            &\le (2\mu n) \exp(-c_\rho n^{3/2}).
        \end{aligned}
    \end{equation}
    
    It remains to estimate the contribution to \eqref{eq:maxunionbound} from the terms in the range $T \le t \le \rho n^{1/2}$. Here we must balance both factors of the terms. From \eqref{eq:lowertailExpldp} and \eqref{eq:lefttail_revisit}, the relevant exponent is
    \begin{multline*}
        \PP(L(\gamma_0) \ge a)\PP(L_n \le a + 1) = \exp\bigl(-2 n I(\mu) + \frac{1}{2}t n^{1/3} - \frac{1}{192}t^{3}
        \\+ O(t^2 n^{-1/3}) + O(t^4 n^{-2/3}) + O(\log t) + O(\log n)\bigr)
    \end{multline*}
 The maximum value of the leading exponent occurs at $t_* = 4\sqrt{2} n^{1/6}$, resulting in a value of $-2 n I(\mu) + \frac{4}{3}\sqrt{2 n}$. To obtain the upper bound, split the range into $t=O(n^{1/6})$ and its complement. In the first range the error terms in the preceding display are $O(\log n)$, uniformly on every fixed multiple of $n^{1/6}$. Outside a sufficiently large fixed multiple of $n^{1/6}$, the negative cubic term dominates both the linear term and the displayed errors; below a sufficiently small fixed multiple, the linear term is strictly below the maximum on the $\sqrt n$ scale. Since the mesh in $t$ is $n^{-1/3}$ and there are only $O(n)$ summands, the largest summand controls the sum up to a factor $\exp(O(\log n))$. Consequently,
    \begin{equation}
        \label{eq:middlesum}
        \sum_{a = \ceil{2 n \mu - \rho n^{5/6}}}^{\floor{4n - T n^{1 / 3}}}\PP(L(\gamma_0) \ge a)\PP(L_n \le a + 1) \le \exp\bigl(-2 n I(\mu)+\frac{4}{3}\sqrt{2n} + O(\log n)\bigr).
    \end{equation}

    Between \eqref{eq:tailsum}, \eqref{eq:headsum} and \eqref{eq:middlesum}, we have bounded the right hand side of \eqref{eq:maxunionbound}. Putting the pieces together leads to
    \[
    \max_{\gamma \in \Gamma_n}\PP(\gamma = \gamma^*_n) \le \exp\bigl(-2 n I(\mu) + \frac{4}{3} \sqrt{2n} + O(\log n)\bigr),
    \]
    and after taking logarithms and adding $2 n I(\mu)$, we find the claimed upper bound.
\end{proof}

The lower bound involves a number of moving pieces. In describing an event under which $\cornerpath$ is a geodesic, we must ensure there are no ``shortcuts'' through the bulk.

The bulk passage times we wish to control have their starting points near either $(0, k)$ or $(k, n)$, $0 \le k \le n$. We can imagine that the passage times are heavily dependent when their endpoints are separated at distance smaller than $n^{2/3}$, and are otherwise mostly independent. Thus, we  need only control the equivalent of $O(n^{1/3})$ independent variables.

The first bound we need for shortcuts is stated in the next lemma. Recall that
\[
L^{\bltriangle}(a, b; n) = \max_{0 \le k \le n - (a + b)}L(a, b; a + k, n - a - k).
\]
\begin{lemma}
\label{lem:cornerdomination}
    Let $C_k$ be the weight along the straight path $(0, n - k) \to (0, n)$ and let $L^{\bltriangle}_k = L^{\bltriangle}(1, n - k - 1; n)$ be a point-to-line passage time from a shifted starting point. For each $\epsilon > 0$ there is a constant $c > 0$ such that for $n$ large enough,
    \begin{equation}
        \PP(C_k \ge k \mu \ge L^{\bltriangle}_k \text{ for all } 1 \le k \le n - 1) \ge e^{-n I(\mu) - c n^{1/3 + \epsilon}}.
    \end{equation}
\end{lemma}

\begin{proof}
    The corner and bulk passage times are independent of one another, so we consider two independent events:
    \[
    \set{C_k \ge k \mu \text{ for all } 1 \le k \le n - 1},\qquad \set{L_k^{\bltriangle} \le k \mu \text{ for all } 1 \le k \le n - 1}.
    \]
    The first is a statement about a random walk with exponential increments, equal to the probability that a walk $\tilde{C}_k = C_k - k \mu$, with steps $\omega_x - \mu$, stays non-negative. An elementary lower bound on this probability is
    \[
    \PP(\tilde{C}_k \ge 0 \text{ for all } 1 \le k \le n - 1) \ge \frac{1}{n - 1}\PP(\tilde{C}_{n - 1} \ge 0).
    \]
    To see this, condition on $\tilde{C}_{n - 1} \ge 0$ and perform a cyclic permutation on the steps so that the minimal value comes first. The permuted path stays above zero. Thus under this conditioning there is always at least one permutation bringing our path to one of the form we seek. There are $n - 1$ possible cyclic permutations, and this gives the bound. Now we recall $\eqref{eq:BRthm}$, and get
    \begin{align*}
        \PP(\tilde{C}_{k} \ge 0 \text{ for all } 1 \le k \le n - 1) &\ge \frac{1}{n - 1}e^{-n I(\mu) + O(\log n)}\\ &= e^{-n I(\mu) + O(\log n)}.
    \end{align*}

    Next we look for a lower bound on $\PP(L_k^{\bltriangle} \le k \mu \text{ for all } 1 \le k \le n - 1)$. For $0 \le a < b \le n - 1$, consider 
    \[
        M^{\bltriangle}_{a, b} = \max_{a \le k \le b} \sqrbrac{L^{\bltriangle}_k - k \mu}.
    \]
    The event $\set{M_{a, b}^{\bltriangle} \le 0}$ is decreasing in the weights and thus, by the Harris inequality, for any sequence $1 = a_1 \le a_2 \le \cdots \le a_r = n-1$, we have
    \begin{align}
        \PP(L_k^{\bltriangle} \le k \mu \text{ for all } 1 \le k \le n - 1) &= \PP(M_{a_i, a_{i + 1}}^{\bltriangle} \le 0 \text{ for all } 1 \le i \le r - 1)\notag\\
        &\ge \prod_{i = 1}^{r - 1}\PP(M_{a_i, a_{i + 1}}^{\bltriangle} \le 0).\label{eq:P2LHarris}
    \end{align}
    If $b - a \ll b^{2/3}$, the geodesics for the passage times defining $M^{\bltriangle}_{a, b}$ coalesce at a microscopic distance with high probability and 
    \[
        \PP(M_{a, b}^{\bltriangle} \le 0) \xrightarrow{b \to \infty} \PP(\mathrm{TW}_1 \le 0) > 0.
    \]
    Here $\mathrm{TW}_1$ is the Tracy-Widom GOE distribution. That $(2n)^{-1/3}(L_n^{\bltriangle} - n\mu)$ converges to $\mathrm{TW}_1$ was indicated in \cite{BR:randomperms} and may be found explicitly in \cite[Theorem 3.18]{B:polymers}. Extending this to the maximum over a small interval is a simple use of coalescence estimates found, for example, in  \cite{Z:coalescence}.

    Using this convergence, pick $p > 0$ small and find $K$ such that $k \ge K$ ensures
    \[
        \PP(M_{\floor{k - k^{2/3 - \epsilon}}, k}^{\bltriangle} \le 0) \ge p.
    \]
    Also let $p_0 = \PP(M^{\bltriangle}_{0, K} \le 0) > 0$. We consider dividing $\bbrac{K, n-1}$ into $r = O(n^{1/3 + \epsilon})$ segments with small enough lengths. Specifically, we set $a_0 = K$, and subsequently $a_{i + 1} = a_i + \floor{a_{i}^{2 / 3 - \epsilon}}$, proceeding until the sequence exceeds $n-1$ and we cap it with $a_r = n - 1$. To see why $r = O(n^{1/3 + \epsilon})$, observe that we may lower bound the finite difference equation defining $a_i$ by the differential equation $dy = \frac{1}{2} y^{2/3 - \epsilon}\, dx$, whose solution grows like $C x^{3/(1 + 3\epsilon)}$ and whose inverse is like $c y^{1/3 + \epsilon}$.
    
    From \eqref{eq:P2LHarris} and with this choice of $a_i$,
    \begin{align*}
        \PP(L_k^{\bltriangle} \le k \mu \text{ for all } 1 \le k \le n - 1) \ge p_0 p^{c n^{1/3 + \epsilon}},
    \end{align*}
    which is of the form $\exp( - c n^{1/3 + \epsilon})$. Multiplying this with the probability we got for the $C_k$ gives the claimed lower bound. 
\end{proof}

The other ingredient is showing that we can improve the one-point estimate for point-to-line times found in Theorem~\ref{thm:sharpmd} to an estimate over the maximum of times starting from a few nearby points.

\begin{lemma}
\label{lem:intervallefttail}
    Let $\tilde{M}^{\bltriangle}_{c, d} = \max_{c \le k \le d}\left[L^{\bltriangle}(0, k; n) - (n - k)\mu\right]$. Fix $\epsilon > 0$ and suppose $a \le \floor{n^{2/3 - \epsilon}}$. Let $T,\, \rho > 0$ be as in Theorem~\ref{thm:sharpmd}. When $T \le t \le \rho n^{2/3 - \epsilon}$, we have the interval-to-line estimate
    \[
        \log \PP(\tilde{M}^{\bltriangle}_{0, a} \le -t n^{1/3}) = -\frac{1}{48} t^{3} + O(t^3 n^{-\epsilon / 6}) + O(t^{3/2}).
    \]
    \end{lemma}
\begin{proof}
It suffices to prove the claim with the largest possible $a$, so fix $a = \floor{n^{2/3 - \epsilon}}$. Write $b = \floor{n^{1 - \epsilon}}$ and fix $\delta = \epsilon/6$. Consider the passage time from a point behind our region of interest, $L' = L^{\bltriangle}(-b, -b; n)$. We introduce a gap $\alpha = t n^{1/3 - \delta}$ and define the event:
    \[
        B = \left\{L' - (n + 2b)\mu \le - t (n + 2b)^{1/3} - \alpha\right\}.
    \]
    We will show that on $B$, a small value of $L'$ forces a small value in each $L^{\bltriangle}(0, k; n)$. Put $z_k=(0,k)$ and define the endpoint-excluded time
    \[
        X_k = L(-b,-b;z_k) - \omega_{z_k}.
    \]
    Let $Y_k=L^{\bltriangle}(z_k; n)$ be the point-to-line passage times we are interested in controlling, and define also the events
    \[
        A_k = \left\{Y_k - (n - k)\mu > -t n^{1/3}\right\}.
    \]
    Concatenation through $z_k$ gives $L'\ge X_k + Y_k$. Hence, for $B$ and $A_k$ to occur together, it must be that
    \begin{align*}
        X_k &\le L' - Y_k \\
        &\le (n + 2b)\mu - (n - k)\mu - t (n + 2b)^{1/3} - \alpha + t n^{1/3}\\
        &= (2b + k)\mu - t\bigl((n + 2b)^{1/3} - n^{1/3}\bigr) - \alpha.
    \end{align*}
    Observe that the variables $X_k$ and $Y_k$ are independent. We should centre $X_k$ by the time constant $\mu(b, b+k) = (\sqrt{b} + \sqrt{b + k})^2$, and we may approximate the difference
    \[
    (2b + k)\mu - \mu(b, b + k) = \frac{k^2}{4b} + O\bigl(\frac{k^3}{b^2}\bigr).
    \]
    For $k \le a = \floor{n^{2/3-\epsilon}}$ and $b = n^{1-\epsilon}$, we have $k^2 / b \le n^{1/3 - \epsilon}$. Additionally, $t\brac[\big]{(n+2b)^{1/3} - n^{1/3}} \ge 0$. Substituting these into the constraint yields:
    \[
        X_k - \mu(b, b + k) \le n^{1/3-\epsilon} - \alpha + O(n^{-\epsilon}).
    \]
    Because $\delta = \epsilon/6 < \epsilon$, the shift $\alpha = t n^{1/3-\delta}$ dominates the other term. (Here we need $t \ge T$.) Thus $X_k$ is forced deep into its lower-tail:
    \[
        X_k - \mu(b, b + k) \le - t n^{1/3-\delta} (1 - o(1)).
    \]
    Observe that $t n^{1/3-\delta} \ge T n^{1/3 - \delta} > T b^{1/3}$ and $t n^{1/3-\delta} \le \rho n^{1 - \epsilon - \delta} \le \rho b^{1 - \delta}$. A fluctuation of size $t n^{1/3-\delta}$ thus sits in the range in which \eqref{eq:lefttail} applies (with $n^{1 - \epsilon}$ in place of $n$), and hence
    \[
        \PP(B \cap A_k) \le \exp\left( -c \frac{(t n^{1/3-\delta})^3}{b} \right) \leq \exp\left( -c t^3 n^{\epsilon/2} \right),
    \]
    for some constant $c > 0$. Strictly speaking, the bound in \eqref{eq:lefttail} is valid only for passage times in the diagonal direction and we should prove uniformity across directions. However, we do not need the exact constant and the bound from \cite[Theorem 2]{LR:betadeviations} suffices.

    We now compare this single-path penalty to the rarity of the global event $B$. Let $t'$ be the effective threshold such that $t' (n + 2b)^{1/3} = t (n+2b)^{1/3} + \alpha$. This implies $t' = t + t n^{-\delta}  + O(t n^{-(\delta + \epsilon)})$. All of this can be absorbed into the error when applying the lower-tail estimate for $L'$:
    \begin{align*}
        \log \PP(B) &= -\frac{1}{48} (t')^3 + O((t')^4 (n + 2b)^{-2/3}) + O((t')^{3/2})\\
                    &= -\frac{1}{48} t^3 + O(t^3 n^{-\delta}) + O(t^4 n^{-2/3}) + O(t^{3/2})\\
                    &= -\frac{1}{48} t^3 + O(t^3 n^{-\delta}) + O(t^{3/2}).
    \end{align*}
    The first equality follows from \eqref{eq:sharpp2l}, since $L' \disteq L^{\bltriangle}_{n + 2b}$. In the last equality we used $t n^{-2/3} \le \rho n^{-\epsilon} \le \rho n^{-\delta}$.
    The penalty on $X_k$, which scales as $\exp(-c t^3 n^{\epsilon/2})$, is vastly more severe than the probability of $B$, which scales as $\exp(- t^3 / 48)$ (abbreviating lower order terms). Thus, $\PP(B \cap A_k)$ is overwhelmingly smaller than $\PP(B)$. Applying the union bound over $k \in [0, a]$ with $a = n^{2/3-\epsilon}$:
    \[
        \PP\left(B \cap \bigcup_{k=0}^{a} A_k\right) \le \sum_{k = 0}^a \PP(B \cap A_k) \le (a+1) \exp\left(-c t^3 n^{\epsilon/2}\right) = o(1)\PP(B).
    \]
    Taking complements relative to $B$:
    \[
        \PP\left(B \cap \bigcap_{k=0}^{a} A_k^c\right) = \PP(B) (1 - o(1)).
    \]
    Since $\bigcap_{k=0}^a A_k^c = \left\{\tilde{M}^{\bltriangle}_{0,a} \le -t n^{1/3}\right\}$, we obtain the lower bound:
    \[
        \log \PP(\tilde{M}^{\bltriangle}_{0,a} \le - t n^{1/3}) \ge \log \PP(B) + o(1)
        = -\frac{1}{48} t^3 + O(t^3 n^{-\delta}) + O(t^{3/2}).
    \]
    For the reverse bound, the term $k = 0$ gives
    \[
        \PP(\tilde{M}^{\bltriangle}_{0,a} \le - t n^{1/3})
        \le \PP(L^{\bltriangle}(0, 0; n) \le n \mu - t n^{1/3}).
    \]
    Combining \eqref{eq:sharpp2l} with the lower bound and recalling $\delta = \epsilon / 6$, we obtain
    \[
        \log \PP(\tilde{M}^{\bltriangle}_{0,a} \le - t n^{1/3})
        = -\frac{1}{48} t^3 + O(t^3 n^{-\epsilon / 6}) + O(t^{3/2}).
    \]
\end{proof}
\begin{proof}[Proof of the lower bound in Theorem~\ref{thm:probability}]
We demonstrate an event on which we see $\cornerpath = \gamma_n^*$ and lower bound its probability. Lemma~\ref{lem:cornerdomination} controls the probability needed to eliminate shortcuts from the corner through the bulk, while Lemma~\ref{lem:intervallefttail} lets us loosen the tilting on the initial and final segments of $\cornerpath$ to improve our probability.

Recall the notation from Lemma~\ref{lem:cornerdomination}, where $C_k$ is the weight along $(0, n - k) \to (0, n)$ and $L_k^{\bltriangle} = L^{\bltriangle}(1, n - k - 1; n)$. Also introduce $\tilde{C}_k$ for the weight along $(0, n) \to (k, n)$, and the line-to-point passage time $\tilde{L}_k^{\trtriangle} = L^{\trtriangle}(n + 1; k + 1, n - 1)$. Set $a = \floor{n^{3/5}}$ and choose $0 < \epsilon < 1/15$. (Any exponent in $(\frac{1}{2}, \frac{2}{3})$ can be made to work equally well.) Using $\alpha > 0$ as a value to be determined later, consider events
\begin{align*}
    A_1 &= \set{C_k \ge k \mu \ge L^{\bltriangle}_k \text{ for all } 1 \le k \le n - a - 1},\\
    A_2 &= \set{C_{k} \ge k \mu - \alpha\sqrt{n}, \text{ for all } n - a \le k \le n - 1},\\
    A_3 &= \set{L^{\bltriangle}_k \le k \mu - \alpha\sqrt{n} \text{ for all } n - a \le k \le n - 1}.
\end{align*}
Set $A = A_1 \cap A_2 \cap A_3$. Observe that on event $A$, the passage time of any path from the origin to level $n$ is dominated by the vertical path $(0, 0) \to (0, n)$. Consider also the corresponding event $\tilde{A}$ on the north-east half of the square less the antidiagonal, involving the variables $\tilde{C}_k$ and $\tilde{L}_k^{\trtriangle}$, which guarantees that paths from level $n + 1$ to $(n, n)$ are dominated by $(1, n) \to (n, n)$. These events are independent and on their intersection we have $\gamma_n^* = \cornerpath$.

We set about putting a lower bound on $\PP(A)$, looking at its constituent events in turn. Lemma~\ref{lem:cornerdomination} gives the bound on $A_1$, which is
\begin{equation}
    \label{eq:A1lowerbound}
    \PP(A_1) \ge \exp\brac{-(n - a)I(\mu) - c n^{1/3 + \epsilon}}.
\end{equation}
For $A_2$, it is strongly correlated with $A_1$ and so we write
\(
\PP(A_1 \cap A_2) = \PP(A_1)\PP(A_2 \mid A_1).
\) 
Conditioning on $A_1$ in particular means conditioning on $C_{n - a - 1} \ge (n - a - 1)\mu$, which is greater (for all but small $n$) than the value $(n - a)\mu - \alpha\sqrt{n}$ we want for $C_{n - a}$. We can get a lower bound by looking at the (unconditioned) probability of the event
\[
    \set{C_{n - a + k} - C_{n - a - 1} \ge (k + 1)\mu - \alpha\sqrt{n}, \text{ for all } 0 \le k \le a - 1}.
\]
This is just the probability of a random walk remaining above a linear boundary. Using the Bahadur-Rao theorem of \eqref{eq:BRthm} with a linear approximation of $I$ near $\mu$, it is then not difficult to get
\begin{equation}
\label{eq:A2lowerbound}
\PP(A_2 \mid A_1) \ge \exp\brac[\big]{- a I(\mu) + (1 - \delta)\alpha \sqrt{n} / 2 + o(\sqrt{n})}.
\end{equation}
Finally, $A_3$ and $A_1 \cap A_2$ are decreasing in the bulk weights and increasing in the weights along $\cornerpath$ (trivially in the case of $A_3$). The Harris inequality applied to the two sets of weights shows that they are positively correlated, and so
\[
\PP(A_3 \mid A_1 \cap A_2) \ge \PP(A_3).
\]
By Lemma~\ref{lem:intervallefttail}, the unconditioned probability of $A_3$ has
\begin{equation}
    \label{eq:A3lowerbound}
    \PP(A_3) \ge \exp\brac[\bigg]{- \frac{1 + \delta}{48} \alpha^3 \sqrt{n}},
\end{equation}
which holds for any $\delta > 0$ and $n$ large enough. We may now combine \eqref{eq:A1lowerbound}, \eqref{eq:A2lowerbound} and \eqref{eq:A3lowerbound} and obtain
\begin{equation*}
    \PP(A) \ge \exp\brac[\big]{- n I(\mu) - c n^{1/3 + \epsilon} + (1 - \delta)\alpha \sqrt{n} / 2 - \frac{1 + \delta}{48} \alpha^3 \sqrt{n} + o(\sqrt{n})}.
\end{equation*}
At this point we are free to choose $\alpha > 0$ to give a good bound on the right. It is enough to take $\alpha = 2\sqrt{2}$ and recall $\epsilon < 1/15 < 1/6$, for which the expression becomes
\[
    \PP(A) \ge \exp\brac[\big]{- n I(\mu) + \frac{2}{3}(1 - 2\delta)\sqrt{2n} + o(\sqrt{n})}.
\]
The same can be done for $\tilde{A}$, and we are left with the lower bound
\begin{align*}
    \PP(\cornerpath = \gamma^*_n) &\ge \PP(A  \cap \tilde{A})\\
    &\ge \exp\brac[\big]{- 2n I(\mu) + \frac{4}{3}(1 - 2\delta)\sqrt{2n} + o(\sqrt{n})}.
\end{align*}
Taking $n \to \infty$ and then $\delta \to 0$ gives
\begin{equation}
    \label{eq:cornerproblower}
    \liminf_{n \to \infty} \frac{\log \PP (\cornerpath = \gamma^*_n) + 2n I(\mu)}{\sqrt{2 n}} \ge \frac{4}{3}.
\end{equation}
This matches the upper bound we obtained earlier.
\end{proof}

We make some observations on the proof of the lower bound which will lead into Corollary~\ref{cor:staircases}. For us to have $\cornerpath = \gamma^*_n$, it is not enough for the corner path to beat a single path going through the bulk (that is, to have $L(\cornerpath) \ge L(1, 0; n, n - 1)$). Rather, each segment of the corner path must beat the bulk path connecting that segment's endpoints.

\begin{figure}
    \centering
    \resizebox{0.6\textwidth}{!}{\input{fig_staircase}}
    \caption{The staircase $\gamma^k_n$, shown with $k = 3$. We consider a half-space geodesic lying below the staircase and many geodesics going through the bulk of the staircase's steps. The cyan paths are unrestricted point-to-point geodesics, while the outermost magenta paths are half-space geodesics above the diagonal. The geodesics are spaced so that they do not meet, except on an event of negligible probability. The segments of the staircase which border the interior of the bulk geodesics (the shaded regions) must have large weights to successfully discourage shortcuts. These segments comprise all but the two ends of the path. As we scale $n$ and the number of bulk geodesics grows polynomially, this effect suppresses the overall probability of the event $\set{\gamma^k_n = \gamma^*_n}$ and pushes it below our value for $\set{\cornerpath = \gamma^*_n}$.}
    \label{fig:staircase}
\end{figure}
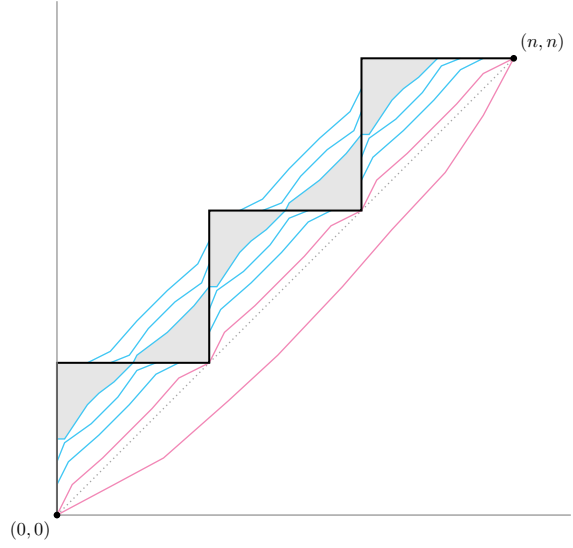

To make the corner path competitive with typical bulk paths, we must essentially tilt the weights so that their mean is $\mu$, incurring a cost of $e^{-2n I(\mu)}$. The $\Omega(\sqrt{n})$ gain we see in Theorem~\ref{thm:probability} comes from the ability to weaken the weights in an $n^{2/3}$ interval at the endpoints of the path. The trade-off is that we must also weaken the bulk paths between these endpoints. Lemma~\ref{lem:intervallefttail} shows that when the endpoints are this close, this is essentially the same as weakening a single passage time, $L(1, 0; n, n - 1)$. By optimising the benefit of weakening the weights along $\cornerpath$ against the cost of weakening the bulk passage time, we obtain the lower bound.

Imagine weakening weights on $\cornerpath$ nearer the corner, say around $(0, \frac{n}{2})$. These weights are needed in beating $L(1, k; n - k, n - 1)$ for all $k = 0, \dots, \frac{n}{2}$. Even with the consideration of coalescence, this requires around $n^{1/3}$ essentially independent lower-tail events. Thus we see that smaller weights anywhere except near the endpoints incur an inordinate cost. 

In the case of a staircase path $\gamma^k_n$, the weights furthest from the corners, and hence easiest to shift downward, live near the two endpoints of the path. We see very little gain from the interior blocks of the paths.

For fear that the idea of the proof of Corollary~\ref{cor:staircases} becomes lost in the details, the core of the argument is depicted in Figure~\ref{fig:staircase}.

\begin{proof}[Proof of Corollary~\ref{cor:staircases}]
    The proof is in essence an implementation of the idea described above, but to carry it out formally we must do quite a bit of tedious bookkeeping. For simplicity assume $n$ is a multiple of $k$, write $n = m k$, and assume that $Q = m^{1/10}$ is an integer and that $m$ is even. Consider points along the path
    \begin{align*}
        u^i_j &= (j m, j m + i m^{9/10}),\\
        v^i_j &= \bigl((j + 1) m - i m^{9 / 10}, (j + 1)m \bigr),
    \end{align*}
    where $i \in \bbrac{0, Q}$ and $j \in \bbrac{0, k - 1}$. Use these to define weights along segments of the staircase
    \begin{align*}
        D^i_j &= L(u^i_j + e_2 \to u^{Q}_j \to v^i_j - e_1),\\
        \tilde{D}^i_j &= L(v^i_j \to v^0_j \to u^i_{j + 1}).
    \end{align*}
    Here the ranges are $i \in \bbrac{0, Q - 1}$, and $j \in \bbrac{0, k - 1}$ for the first set of variables and $j \in \bbrac{0, k - 2}$ for the second. Also define bulk passage times
    \begin{align*}
        L^i_j &= L(u^i_j + e_1; v^i_j - e_2),\\
        \tilde{L}^i_j &= L(v^i_j - e_1 + e_2; u^i_{j + 1} - e_1 + e_2).
    \end{align*}
    For these we want to take $i \in \bbrac{1, Q - 1}$ but $j$ as before. For $i = 0$, define specially half-space times
    \[
        L^0_j = L^{\tltriangle}_j = L^{\tltriangle}(u^0_j + e_1, v^0_{j} - e_2).
    \]
    Let $L^{\brtriangle} = L^{\brtriangle}(2, 0; n, n - 2)$ be a half-space time spanning the whole breadth of $\gamma^k_n$. A diagram depicting some of these variables is given in Figure~\ref{fig:staircase_notation}.

    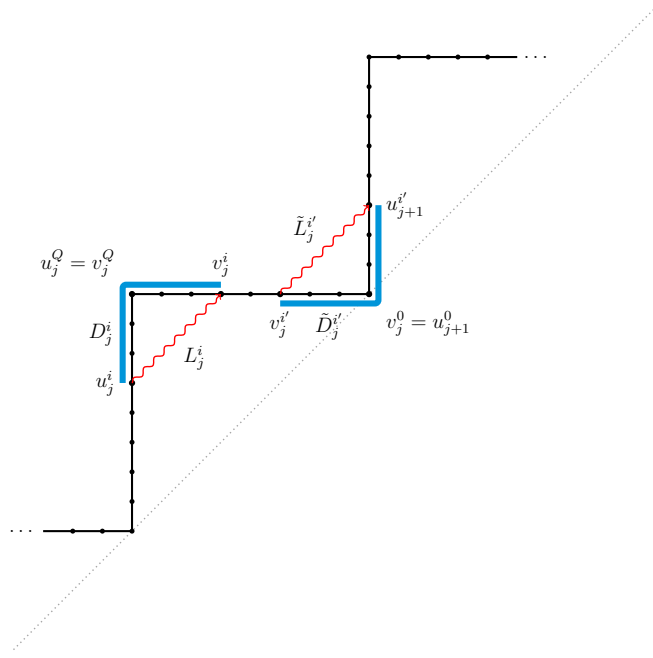
\begin{figure}
        \centering
        \resizebox{0.7\textwidth}{!}{\input{fig_staircase_notation.tex}}
        \caption{Some of the notation introduced for points on the staircase $\gamma^k_n$, and passage times on and near it. In the diagram $i' < i$. Note that the points $u^i_j$ move upward as $i$ increases, while $v^i_j$ move leftward.}
        \label{fig:staircase_notation}
    \end{figure}
    
    The various offsets are chosen so that passage times with different indices $j$ are completely independent, and so that $L^{\brtriangle}$ is independent of all other passage times. Additionally, each is independent of the weights on $\gamma^k_n$. We will go a step further and treat \emph{all} of the times defined above as independent. This is justified in Appendix~\ref{app:transflucts}, wherein we show that geodesics of length at most $n$ with endpoints separated by $m^{9/10}$ will be disjoint with probability at least $1 - \exp(- c m^{7 / 10})$. As the probabilities involving the passage times will be on the order of $\exp(- c \sqrt{n})$, we may treat the effects of interaction of geodesics as negligible error terms. Proposition~\ref{prop:highprobindependence} contains the precise estimate justifying this simplification.

    In order for $\gamma^k_n$ to be the geodesic, it must in particular be true that subpaths have higher weight than shortcuts across their corners. We have an inclusion of events
    \begin{align}
    \label{eq:staircaseinclusion}
        \set{\gamma^k_n = \gamma^*_n} &\subseteq \set{L(\gamma^k_n) \ge L^{\brtriangle}} \cap \set{D^0_0 \ge L^{\tltriangle}_0} \cap \set{D^0_{k - 1} \ge L^{\tltriangle}_{k - 1}}\\ 
        &\qquad \cap \bigcap_{j = 0}^{k - 1}\bigcap_{i = 1}^{Q - 1}\set{D^i_j \ge L^i_j} \cap \bigcap_{j = 0}^{k - 2}\bigcap_{i = 1}^{Q - 1}\set{\tilde{D}^i_j \ge \tilde{L}^i_j}.\notag
    \end{align}
    We put an upper bound on the probability of the intersection on the right. In the spirit of \eqref{eq:maxunionbound} in the proof of the upper bound in Theorem~\ref{thm:probability}, there is a union bound by a discrete sum of events involving the bulk passage times and corner sums. The rest of the proof proceeds by estimating the maximum term in this sum, which is the maximiser of the joint density of the variables constrained to the events in \eqref{eq:staircaseinclusion}. By an appeal to Laplace's method, this dominates the contributions of the other terms. We do not show carefully how to bound the off-maximum terms for reasons of length, but this may be carried out in the same way as we did for \eqref{eq:maxunionbound}, just with more factors.

    Before doing that, we introduce yet more notation. For $i \in \bbrac{0, Q - 1}$, set
    \begin{align*}
        r^i_j &= m^{-1/3}\brac[\Big]{L\brac[\big]{u^i_j + e_2 \to u^{i + 1}_j} -  \mu m^{9 / 10}},\\
        \tilde{r}^i_j &= m^{-1/3}\brac[\Big]{L\brac[\big]{v^{i + 1}_j \to v^i_j - e_1} -  \mu m^{9 / 10}}.
    \end{align*}
    These are normalised weights along vertical and horizontal, respectively, subsegments of the staircase path. Observe that the suitably normalised forms of the $D^i_j$ are just sums
    \[
        m^{-1/3} \brac[\big]{D^i_j - 2\mu (m - i m^{9 / 10})} = \sum_{\iota = i}^{Q - 1} \sqrbrac{r^{\iota}_j + \tilde{r}^{\iota}_j},
    \]
    and also that
    \[
        m^{-1/3} \brac[\big]{\tilde{D}^i_j - \omega_{v^0_j}- 2\mu i m^{9 / 10}} = \sum_{\iota = 0}^{i - 1} \sqrbrac{\tilde{r}^{\iota}_j + r^{\iota}_{j + 1}}.
    \]
    We choose subsequently to ignore the correction by $\omega_{v^0_j}$. Introduce as well
    \begin{align*}    
        l^i_j &= m^{-1/3}\bigl(L^i_j - 2 \mu (m - i m^{9 / 10})\bigr),\quad 
        \tilde{l}^i_j = m^{-1/3}\bigl(\tilde{L}^i_j - 2 \mu i m^{9 / 10}\bigr),
    \end{align*}
    to be normalised versions of the bulk passage times. Similarly for the  half-space times, let $l^{\tltriangle}_j = m^{-1/3}(L^{\tltriangle}_j - 2\mu m)$ and $l^{\brtriangle} = m^{-1/3}(L^{\brtriangle} - 2\mu n)$. We have the following bounds for the tails of each of these variables, wherein we ignore lower order terms: 
    \begin{align*}
        \log \PP(r^i_j \ge - \lambda) &\le - m^{9 / 10}I(\mu) + \frac{\lambda m^{1/3}}{2}\\
        \log \PP(l^i_j \le - \lambda) &\le - c_f \lambda^3\\
        \log \PP(l^{\tltriangle}_j \le - \lambda) &\le - c_h \lambda^3\\
        \log \PP(l^{\brtriangle} \le - \lambda) &\le - \frac{c_h}{k} \lambda^3.
    \end{align*}
    Here $c_f = \frac{1}{192}$ is the constant appearing in the moderate deviation estimate of Theorem~\ref{thm:sharpp2pmd} and $c_h = \frac{1}{384}$ is the half-space equivalent from Theorem~\ref{thm:sharpmd}. We give our inequalities for $\lambda > 0$, as there is no probability gain to be had by letting the passage times be bigger than typical. The passage times across shorter distances have lower variance and more drastic tail decay, but we choose to use a uniform bound.

    There is no gain to be had in taking the values of $r^i_j$, $\tilde{r}^i_j$ positive (and tilting the weights on $\gamma^k_n$ beyond a mean of $\mu$), so we may always imagine them negative. Recall the events in \eqref{eq:staircaseinclusion} and condition on the values $r^i_j = - x^i_j$ and $\tilde{r}^i_j = -\tilde{x}^i_j$, for all $i,\, j$ and $x^i_j,\, \tilde{x}^i_j$ being positive. Write $\PP_{x}$ for the probability under this conditioning. The log-probabilities of the various events have:
    \begin{align*}
        \log \PP_x\brac{L(\gamma^k_n) \ge L^{\brtriangle}} &\le - \frac{c_h}{k}\brac[\bigg]{\sum_{j = 0}^{k - 1}\sum_{i = 0}^{Q - 1}\sqrbrac{x^i_j + \tilde{x}^i_j}}^3\\
        \log \PP_x\brac{D^0_0 \ge L^{\tltriangle}_0,\, D_{k - 1}^0 \ge L^{\tltriangle}_{k - 1}} &\le - c_h \brac[\bigg]{\sum_{i = 0}^{Q - 1}\sqrbrac{x^i_0 + \tilde{x}^i_0}}^3
        \\      &\qquad\quad - c_h \brac[\bigg]{\sum_{i = 0}^{Q - 1}\sqrbrac{x^i_{k - 1} + \tilde{x}^i_{k - 1}}}^3\\
        \log \PP_x\brac{D^i_j \ge L^i_j \text{ for all } i,\, j} &\le - c_f \sum_{j = 0}^{k - 1}\sum_{i = 1}^{Q - 1}\brac[\bigg]{\sum_{\iota = i}^{Q - 1}\sqrbrac{x^\iota_{j} + \tilde{x}^\iota_{j}}}^3\\
        \log \PP_x\brac{\tilde{D}^i_j \ge \tilde{L}^i_j \text{ for all } i,\, j} &\le - c_f\sum_{j = 0}^{k - 2}\sum_{i = 1}^{Q - 1}\brac[\bigg]{\sum_{\iota = 0}^{i - 1}\sqrbrac{\tilde{x}^\iota_{j} + x^\iota_{j + 1}}}^3.
    \end{align*}
    The log-cost of this conditioning, on the other hand, is 
    \[
        - 2 n I(\mu) + \frac{m^{1/3}}{2}\sum_{j = 0}^{k - 1}\sum_{i = 0}^{Q - 1} \sqrbrac{x^i_j + \tilde{x}^i_j}.
    \]
    Recall also that it is enough to consider these events as independent. To estimate the probability of the intersection in \eqref{eq:staircaseinclusion} then, we ought to optimise the sum of these log-probabilities over the values of the $x^i_j$, $\tilde{x}^i_j$. The $- 2 n I(\mu)$ contribution remains constant, so all told, the target function we should optimise is
    \begin{multline}
    \label{eq:staircasetarget}
        F(x, \tilde{x}) := \frac{m^{1/3}}{2}\sum_{j = 0}^{k - 1}\sum_{i = 0}^{Q - 1} \sqrbrac{{x^i_j} + {\tilde{x}^i_j}}- \frac{c_h}{k}\brac[\bigg]{\sum_{j = 0}^{k - 1}\sum_{i = 0}^{Q - 1}\sqrbrac{x^i_j + \tilde{x}^i_j}}^3 \\
            - c_h \brac[\bigg]{\sum_{i = 0}^{Q - 1}\sqrbrac{x^i_0 + \tilde{x}^i_0}}^3 - c_h \brac[\bigg]{\sum_{i = 0}^{Q - 1}\sqrbrac{x^i_{k - 1} + \tilde{x}^i_{k - 1}}}^3 \\
            - c_f \sum_{j = 0}^{k - 1}\sum_{i = 1}^{Q - 1}\brac[\bigg]{\sum_{\iota = i}^{Q - 1}\sqrbrac{x^\iota_{j} + \tilde{x}^\iota_{j}}}^3 - c_f \sum_{j = 0}^{k - 2}\sum_{i = 1}^{Q - 1}\brac[\bigg]{\sum_{\iota = 0}^{i - 1}\sqrbrac{\tilde{x}^\iota_{j} + x^\iota_{j + 1}}}^3.
    \end{multline}
    It is something of a calculation to get the asymptotic bound on $F$, and we split this off into Lemma~\ref{lem:targetbound}. We find eventually that 
    \[
        \limsup_{n \to \infty}\frac{\max_{x, \tilde{x} \ge 0}F(x, \tilde{x})}{\sqrt{2 n}} \le \frac{1}{3\sqrt{3c_h(k+4)}}.
    \]
    Carrying out the Laplace method calculation beginning from \eqref{eq:staircaseinclusion} then, one arrives at
    \[
        \limsup_{n \to \infty} \frac{\log \PP (\gamma^k_n = \gamma^*_n) + 2n I(\mu)}{\sqrt{2 n}} \le \frac{1}{3\sqrt{3c_h(k+4)}}.
    \]
    This is to be compared with the lower bound for $\cornerpath$ we obtained in \eqref{eq:cornerproblower}, of $4 / 3$. The gain for the staircase will be strictly less than that for $\cornerpath$ provided $k \ge 5$.
    \end{proof}
        
    \begin{lemma}
    \label{lem:targetbound}
    Let $F(x, \tilde{x})$ be as defined in \eqref{eq:staircasetarget}, with $x = (x^i_j)_{i \in \bbrac{0, Q - 1}, j \in \bbrac{0, k - 1}}$ a collection of positive reals, and $\tilde{x}$ the same. Then
    \[
        \limsup_{n \to \infty}\frac{\max_{x, \tilde{x} \ge 0}F(x, \tilde{x})}{\sqrt{2 n}} \le \frac{1}{3\sqrt{3c_h(k+4)}}.
    \]
    \end{lemma}
    \begin{proof}
    We can see from the expression for $F$ that the variables $x^i_j$ and $\tilde{x}^i_j$ which live closer to the corners will appear in more terms of the sums. It stands to reason that we should optimise $F$ by increasing the variables which are furthest from any corners, namely those near the endpoints of $\gamma^k_n$. 
    
    Let us try to reorganise the expression to make this point clearer. For simplicity of notation, we assume that $Q$ is even. Split the vertical and horizontal segments of $\gamma^k_n$ into two equal subsegments and define
    \begin{gather*}
        p_j = \sum_{i = 0}^{Q / 2 - 1} x^i_j,\qquad
        q_j = \sum_{i = Q / 2}^{Q - 1} x^i_j,\\
        s_j = \sum_{i = Q / 2}^{Q - 1} \tilde{x}^i_j,\qquad
        t_j = \sum_{i = 0}^{Q / 2 - 1} \tilde{x}^i_j.
    \end{gather*}    
    Writing the sums of the $x^i_j$ and $\tilde{x}^i_j$ in terms of the new variables, discarding contributions from bulk passage times near the corners, and replacing cubes of sums by sums of cubes, we get the following bound on $F$:
    \begin{multline*}
        F \le \frac{m^{1/3}}{2} \sum_{j = 0}^{k - 1} (p_j + q_j + s_j + t_j) - \frac{c_h}{k}\brac[\bigg]{\sum_{j = 0}^{k - 1}(p_j + q_j + s_j + t_j)}^3 \\
            - c_h\sum_{j \in \set{0, k - 1}}(p_j^3 + q_j^3 + s_j^3 + t_j^3)\\
            - \frac{m^{1/10}c_f}{2}\sum_{j = 0}^{k - 1}(q_j^3 + s_j^3) - \frac{m^{1/10}c_f}{2}\sum_{j = 0}^{k - 2}(t_j^3 + p_{j + 1}^3).
    \end{multline*}
    When the variables are grouped together and further negative terms removed, this becomes
    \begin{align*}
        F &\le \frac{m^{1/3}}{2} \sum_{j = 0}^{k - 1} (p_j + q_j + s_j + t_j) \\
            &\quad - \frac{m^{1/10} c_f}{2}\sum_{j = 1}^{k - 2} (p_j^3 + q_j^3 + s_j^3 + t_j^3)  \\
            &\quad - \frac{m^{1/10} c_f}{2} (q_0^3 + s_0^3 + t_0^3 + p_{k-1}^3 + q_{k-1}^3 + s_{k-1}^3)\\
            &\quad - \frac{c_h}{k} (p_0 + t_{k-1})^3 - c_h (p_0^3 + t_{k-1}^3).
    \end{align*}
    In trying to maximise this expression, we see quite clearly that every variable except $p_0$ and $t_{k - 1}$ is hindered by large and growing negative coefficients in its cubic term. The unaffected variables are exactly those corresponding to the initial and final segments of the path. For large $m$, the maximising choice of the variables will be small except for $p_0$ and $t_{k - 1}$. 

    To be more explicit, for $A, B>0$,
    \[
        \max_{x \ge 0}(Ax-Bx^3)=\frac{2}{3\sqrt{3}}A^{3/2}B^{-1/2}.
    \]
    Applying this with $A=m^{1/3}/2$ and $B=m^{1/10}c_f/2$ to each of the remaining $O(k)$ variables in the preceding upper bound gives a total contribution $O(m^{9/20}) = o(\sqrt m)$. Let 
    \[
        \begin{aligned}
            F' &:= \frac{m^{1/3}}{2}\bigl(p_0 + t_{k-1}\bigr) - \frac{c_h}{k} \bigl(p_0 + t_{k-1}\bigr)^3 \\
               &\qquad - c_h \bigl(p_0^3 + t_{k-1}^3\bigr)
        \end{aligned}
    \]
    be the bound on $F$ with the growing negative terms obviated. Since $k$ is fixed and $\max F'$ is of order $\sqrt m$, it follows that \[
        \begin{gathered}
            \max_{x, \tilde{x} \ge 0} F \le \max_{p, q, s, t \ge 0} F' + O(m^{9/20}),\\
            \limsup_{m \to \infty} \frac{\max_{x, \tilde{x} \ge 0} F}{\max_{p, q, s, t \ge 0} F'} \le 1.
        \end{gathered}
    \]
    At this point we can explicitly find the maximum of the function. It is
    \[
        \begin{gathered}
            p_0^3 + t_{k-1}^3 - \frac{1}{4}(p_0 + t_{k-1})^3
            = \frac{3}{4}(p_0 + t_{k-1})(p_0 - t_{k-1})^2 \ge 0,\\
            \begin{aligned}
                \max_{p, q, s, t \ge 0} F'
                &= \max_{p_0 = t_{k-1} \ge 0} F'\\
                &= \max_{p_0 \ge 0}\set[\bigg]{m^{1/3} p_0 - \frac{2c_h(k + 4)}{k} p_0^3}\\
                &= \frac{\sqrt{2km}}{3\sqrt{3c_h(k+4)}} =\frac{\sqrt{2n}}{3\sqrt{3c_h(k+4)}} .
            \end{aligned}
        \end{gathered}
    \]        
    \end{proof}

\section{Proof of Theorem~\ref{thm:smallpmono}}
\label{sec:proofsmallpmono}
The proof of our last claim, namely the monotonicity of probabilities for small $p$, reduces to counting certain configurations of points and is largely combinatorial. 

Rather than random weights, imagine placing $k$ dots on the grid $\bbrac{0, n}^2$. For a path $\gamma \in \Gamma_n$ and integers $a + b = k$, let $C_\gamma^{a, b}$ be the number of configurations of dots with $a$ of them lying on $\gamma$ and $b$ lying off $\gamma$, and such that $\gamma$ is a geodesic. When the dots are placed uniformly at random (without overlap), then the probability of $\gamma$ being a geodesic is proportional to $\sum_{a = 0}^k C_\gamma^{a, k - a}$.

\begin{lemma}
\label{lem:trivialcounts}
    With $C_{\gamma}^{a, b}$ as above and $k \ge 1$, we have
    \[
        C_\gamma^{0, k} = 0, \quad C_\gamma^{k, 0} = \binom{2n + 1}{k}.
    \]
    When $k = 0$, we have simply $C_\gamma^{0, 0} = 1$.
\end{lemma}

From these values we see that the shape of $\gamma$ is irrelevant when $k = 0, 1$. When $k = 2$, the shape also fails to enter into $C_\gamma^{2, 0}$ and $C_\gamma^{0, 2}$. The value of $C_\gamma^{1, 1}$ is the first which varies with $\gamma$. It admits a representation as a discrete area integral, which we explain now.

Take paths $\gamma,\, \pi \in \Gamma_n$ and let $D(\gamma, \pi)$ be the signed region between them. By this we mean that $D(\gamma, \pi)$ is formed by the boxes in the interior of $\gamma \cup \pi$. A component takes positive sign if $\gamma$ forms the south-east boundary and negative if it forms the north-west. For convenience, we identify a box with the vertex in its north-west corner. For this reason, the points on the north-west boundary are thought to be in $D(\gamma, \pi)$, while the south-east boundary and points in $\gamma \cap \pi$ are not. If $(x, y)$ is a point in a signed region, define $\sgn_{D}(x, y) = \pm 1$ depending on the sign of the component.

For example, let $\gamma$ be the path 
\[
(0, 0) \to (0, n - 2) \to (2, n - 2) \to (2, n) \to (n, n),
\]
and $\pi$ the corner path $\cornerpath$. Then $D(\gamma, \pi)$ has a single component with positive sign,
\[
    D(\gamma, \pi) = \set{(0, n), (1, n), (0, n - 1), (1, n - 1)}.
\]
A more complicated, pictorial example is shown in Figure~\ref{fig:area}.

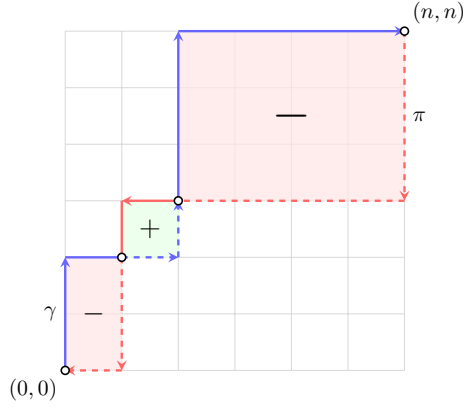
\begin{figure}
    \centering
    \resizebox{0.5\textwidth}{!}{\input{fig_area}}
    \caption{An illustration of the signed area $D(\gamma, \pi)$ between two paths. The dotted lines and holes denote vertices on the boundary which are not included. The interior vertices and remaining boundary vertices are included.}
    \label{fig:area}
\end{figure}

\begin{lemma}
\label{lem:areaformula}
Suppose $\gamma,\, \pi \in \Gamma_n$ and write $D = D(\gamma, \pi)$. Then 
\[
    C_\gamma^{1, 1} - C_\pi^{1, 1} = - 2 \sum_{(x, y) \in D} \sgn_D(x, y)\sqrbrac{y - x - 1}.
\]
\end{lemma}
\begin{proof}
    Write $\gamma = (\gamma^0, \dots, \gamma^{2n})$ and $\gamma^i = (\gamma^i_1, \gamma^i_2)$. Suppose the weight is at $\gamma^i$. Then $\gamma$ will be a (non-unique) geodesic so long as the two weights are not connected by an up-right path. So the other weight can be in one of $\gamma^i_1 (n - \gamma^i_2) + (n - \gamma^i_1) \gamma^i_2$ locations. Hence we have the expression
    \begin{equation*}
        C_\gamma^{1, 1} = \sum_{i = 0}^{2n}\sqrbrac{\gamma^i_1 (n - \gamma^i_2) + (n - \gamma^i_1) \gamma^i_2}.
    \end{equation*}
    We want to compare this to $C_{\pi}^{1, 1}$, where $\pi$ is another path. To write their difference in a more suggestive form, define
    \begin{equation*}
        \mathcal{L}(x, y) = x (n - y) + (n - x) y = n(x + y) - 2 x y.
    \end{equation*}
    Also define $x(t) = \gamma^t_1$ when $0 \le t \le 2n$, and $x(t) = \pi_1^{4n- t}$ when $2n + 1 \le t \le 4n$. We define $y(t)$ similarly. Then $(x(t), y(t))$ traces out a closed loop formed by the union of $\gamma$ and $\pi$, and we may write 
    \begin{equation}
    \label{eq:C11Diff}
        C_{\gamma}^{1, 1} - C_\pi^{1, 1} = \sum_{t = 0}^{4n-1}\mathcal{L}(x(t), y(t))\brac{\Delta x + \Delta y}.
    \end{equation}
    We recognise this as a discrete action functional with respect to the Lagrangian $\mathcal{L}$.

    We are fortunate in that we can greatly simplify \eqref{eq:C11Diff} using the discrete Green theorem for the one-form $\mathcal L\,(dx+dy)$. Its scalar discrete curl is $-2(y - x - 1)$, and Green's theorem immediately produces the claimed formula.
\end{proof}

Although not important for our purposes, this can be turned into an explicit formula for $C_\gamma^{1, 1}$.
\begin{lemma}
With $D_\gamma = D(\gamma, \cornerpath)$,
\[
    C_\gamma^{1, 1} = n^3 - 2 \sum_{(x, y) \in D_\gamma} \sqrbrac{y - x - 1}.
\]
\begin{proof}
    We can calculate $C_{\cornerpath}^{1, 1}$ directly. If we place one dot on $(0, k)$, then there are $n k$ places in the bulk in which to place the other dot while keeping $\cornerpath$ a geodesic. The picture is symmetric for $(k, n)$, and summing gives $C_{\cornerpath}^{1, 1} = n^3$. Now using the difference formula 
    \[ 
        C_\gamma^{1, 1} = C_{\cornerpath}^{1, 1} -  2 \sum_{(x, y) \in D_\gamma} \sqrbrac{y - x - 1}.
    \]
    Here we used that all the area in $D_\gamma$ is positive. 
\end{proof}
\end{lemma}

\begin{lemma}
\label{lem:kmonotone}
    Take $\gamma \in \Gamma_n$ and let $\hat{\gamma}$ be the result of applying the bump move of Figure~\ref{fig:pathbump} to $\gamma$. Then $C_{\hat{\gamma}}^{1, 1} > C_{\gamma}^{1, 1}$. Similarly if $\hat{\gamma}$ is the result of the unwind move of Figure~\ref{fig:pathunwind}.
\begin{proof}
    We should show using Lemma~\ref{lem:areaformula} that $C_{\hat{\gamma}}^{1, 1} - C_{\gamma}^{1, 1} > 0$. Suppose the corner which is bumped is above the diagonal and thus initially pointed south-east. Then one sees that $D(\hat{\gamma}, \gamma)$ has a single negative component lying above the diagonal. The difference $y - x - 1$ is positive here, so the area formula evaluates to a positive quantity. When the corner is below the diagonal, the area is positive, while $y - x - 1$ is negative. This once again gives a positive difference.

    Following the same idea, suppose $\hat{\gamma}$ is the result of unwinding a portion from below to above the diagonal. Then $D(\hat{\gamma}, \gamma)$ is formed from a single negative region. The portion below the diagonal has its reflection about the diagonal also in $D(\hat{\gamma}, \gamma)$. If $(x, y)$ is the north-west corner of a box above the diagonal with contribution $y - x - 1$, then the reflection has north-west corner $(y - 1, x + 1)$  and opposite contribution $(x + 1) - (y - 1) - 1 = 1 + x - y$. The total contributions of these two areas cancel, and what is left is negative area above the diagonal. Here $y - x - 1$ is positive, and $C_{\hat{\gamma}}^{1, 1} - C_{\gamma}^{1, 1} > 0$.
\end{proof}
\end{lemma}

\begin{proof}[Proof of Theorem~\ref{thm:smallpmono}]
        If $n=1$, there are no nontrivial bump or unwind moves, so the result is immediate. Assume that $n \ge 2$.
        Set $N = (n + 1)^2$. In terms of the counts $C_\gamma^{a, b}$, the probability of $\gamma$ being a geodesic can be decomposed as
        \[
            \PP(\gamma \in \Gamma_n^*) = \sum_{k = 0}^{N}\sum_{a = 0}^{k} p^k (1 - p)^{N - k} C_\gamma^{a, k - a}.
        \]
       If $\hat\gamma$ is obtained from $\gamma$ by one bump or unwind move, Lemma~\ref{lem:trivialcounts} gives
        \begin{equation}
        \label{eq:smallpbound}
        \PP(\hat{\gamma} \in \Gamma_n^*) - \PP(\gamma \in \Gamma^*_n)
        = p^2 (1 - p)^{N - 2} \Delta_{\gamma,\hat\gamma}+R_{\gamma,\hat\gamma}(p),
        \end{equation}
        where
        \[
            \Delta_{\gamma,\hat\gamma}
            =C_{\hat\gamma}^{1,1}-C_\gamma^{1,1} \ge 1,
        \]
        by Lemma~\ref{lem:kmonotone}, and 
        \[
            R_{\gamma,\hat\gamma}(p) = \sum_{k = 3}^{N}\sum_{a = 0}^{k} p^k (1 - p)^{N - k} \sqrbrac{C_{\hat\gamma}^{a, k - a}- C_\gamma^{a, k - a}}
        \]
        is the contribution from configurations with three or more dots. For fixed $k$, the sum $\sum_{a=0}^k C_\gamma^{a,k-a}$ counts a subcollection of the $k$-dot configurations and is therefore at most $\binom{N}{k}$. The remainder term thus satisfies the uniform bound
        \[
            \abs{R_{\gamma,\hat\gamma}(p)}
            \le 2 \sum_{k=3}^N \binom{N}{k} p^k(1-p)^{N-k}
            =2 \PP(\operatorname{Bin}(N,p) \ge 3)
            \le 2 \binom{N}{3} p^3.
        \]
        The last inequality follows by applying Markov's inequality to the third falling factorial of a $\operatorname{Bin}(N,p)$ variable. If
        \[
            0<p\le p_n = \frac{1}{2N} \wedge \frac{1}{8 \binom{N}{3}},
        \]
        then $(1-p)^{N-2} \ge 1- (N - 2) p \ge \frac{1}{2}$, while the remainder is at most $p^2/4$. Hence the right-hand side of \eqref{eq:smallpbound} is strictly positive. This choice also shows explicitly that $p_n$ decays at most polynomially in $n$.
\end{proof}

\appendix
\makeatletter
\@addtoreset{theorem}{section}
\renewcommand{\thetheorem}{\thesection.\arabic{theorem}}
\let\c@proposition\c@theorem
\let\c@corollary\c@theorem
\let\c@lemma\c@theorem
\let\c@assumption\c@theorem
\let\c@remark\c@theorem
\let\theproposition\thetheorem
\let\thecorollary\thetheorem
\let\thelemma\thetheorem
\let\theassumption\thetheorem
\let\theremark\thetheorem
\renewcommand{\theHtheorem}{appendix.\thesection.\arabic{theorem}}
\let\theHproposition\theHtheorem
\let\theHcorollary\theHtheorem
\let\theHlemma\theHtheorem
\let\theHassumption\theHtheorem
\let\theHremark\theHtheorem
\makeatother
\section{Lower-tail large deviations}
\label{app:lowertailldp}
\begin{lemma}\label{lemma:lower-tail}
Assume that the weight distribution satisfies Assumption~\ref{ass:lower-tail}.
Then for every $\epsilon>0$ and $A>1$, if $n$ is sufficiently large,
\[
    \PP\bigl(L_n<dn(\mu-\epsilon)\bigr)<e^{-An}.
\]
\end{lemma}
\begin{proof}
We first establish an exponential lower-tail estimate when paths are restricted
to a fixed-width tube. For $x\in\ZZ^d$ and $m\geq 1$, set
\[
 \mathcal T_R(x,m)
 =\left\{z\in\ZZ^d:x\leq z\leq x+m\1,
   \max_{i,j}\left|(z_i-x_i)-(z_j-x_j)\right|\leq R\right\},
\]
and let $L_{\mathcal T_R}(x,x+m\1)$ be the maximum over directed paths
from $x$ to $x+m\1$ contained in this tube. We claim that, for every
$\eta>0$, there are $R,c>0$ such that, for all sufficiently large $m$,
\begin{equation}\label{eq:fixed-tube-lower-tail}
 \PP\left(L_{\mathcal T_R}(0,m\1)<dm(\mu-\eta)\right)\leq e^{-cm}.
\end{equation}

To prove the claim, choose an integer $\ell$ so large that if
\[
 Z=\max_{\gamma:0\to\ell\1}
       \sum_{z\in\gamma\setminus\{0\}}\omega_z,
\]
then
\[
 \EE Z>d\ell(\mu-\eta/4).
\]
Such an $\ell$ exists because $L_\ell=Z+\omega_0$ and
$\EE L_\ell/(d\ell)\to\mu$. Translate the box $[0,\ell]^d$ by
$i\ell\1$, and denote the corresponding copy of $Z$ by $Z_i$, again
omitting the initial vertex of the block. The variables $Z_i$ are
independent: consecutive boxes meet only at a vertex included in the first
block and omitted from the second. Moreover, $Z_i$ dominates the weight of
any fixed path across its box. Assumption~\ref{ass:lower-tail} therefore
implies that $Z_i$ has a finite negative exponential moment. The exponential
Markov inequality, applied with a sufficiently small parameter, gives
\begin{equation}\label{eq:block-lower-tail}
 \PP\left(\sum_{i=0}^{k-1}Z_i
       <d k\ell(\mu-\eta/2)\right)\leq e^{-c_0k}.
\end{equation}

Write $m=k\ell+r$, where $0\leq r<\ell$, and concatenate optimizing paths
in the $k$ boxes with a fixed directed path across the remaining box. This
concatenated path lies in $\mathcal T_\ell(0,m)$. Apart from the block sums,
its weight consists of at most $d\ell+1$ weights. On the complement of the
event in \eqref{eq:block-lower-tail}, it can have passage time less than
$dm(\mu-\eta)$ only if the sum of these finitely many remaining weights is
less than $-c_1m$. In that case at least one of them is less than $-c_2m$.
Assumption~\ref{ass:lower-tail} bounds this probability by
$C e^{-c_3m^\nu}$. Together with \eqref{eq:block-lower-tail}, this proves
\eqref{eq:fixed-tube-lower-tail} with $R=\ell$.

We now place many disjoint copies of this tube in the box $[0,n]^d$. Apply
\eqref{eq:fixed-tube-lower-tail} with $\eta=\epsilon/4$, and fix the
resulting $R$. Choose
\[
 0<q<1-\frac1\nu,\qquad h=\lfloor n^q\rfloor,
 \qquad m=n-h,
\]
and let $s=2R+1$. For a $(d - 1)$-dimensional lattice
\[
 V_n=(s\ZZ^{d-1})\cap[0,h]^{d-1},
\]
write $a_v=(v_1,\ldots,v_{d-1},0) \in V_n$ for a point on the lattice and $b_v=a_v+m\1$ for its shift in the diagonal direction. The tubes
$\mathcal T_R(a_v,m)$, $v\in V_n$, are pairwise disjoint. Indeed, the
oscillation of the coordinate differences of two distinct vectors $a_v$ is
at least $s>2R$. Consequently, the restricted passage times
\[
 S_v=L_{\mathcal T_R}(a_v,b_v),\qquad v\in V_n,
\]
are independent. Since $|V_n|\geq c_4h^{d-1}$ for all large $n$,
translation invariance and \eqref{eq:fixed-tube-lower-tail} yield
\begin{equation}\label{eq:many-tube-lower-tail}
 \PP\left(S_v<dm(\mu-\epsilon/4)
          \text{ for every }v\in V_n\right)
 \leq \exp\{-c_5n^{1+q(d-1)}\}.
\end{equation}

For each $v$, choose deterministic directed connectors from $0$ to $a_v$
and from $b_v$ to $n\1$. In the first connector sum omit $a_v$, and in the
second omit $b_v$, so that concatenating either connector pair with a path in
the tube counts every vertex exactly once. Each connector contains at most
$dh$ counted vertices. If its weight is less than $-d\epsilon n/4$, then
one of its vertices has weight less than $-\epsilon n/(4h)$. A union
bound over all connector vertices and Assumption~\ref{ass:lower-tail} give
\begin{equation}\label{eq:connector-lower-tail}
 \begin{aligned}
 &\PP\left(\text{some connector has weight less than }
                    -d\epsilon n/4\right)\\
 &\qquad\leq Cn^{qd}\exp\{-c_6n^{\nu(1-q)}\}
 \leq \exp\{-c_7n^{\nu(1-q)}\}.
 \end{aligned}
\end{equation}

Since $h=o(n)$, for all sufficiently large $n$,
\[
 dm(\mu-\epsilon/4)\geq dn(\mu-\epsilon/2).
\]
Thus, if some $S_v$ is at least $dm(\mu-\epsilon/4)$ and neither of its
connectors has weight less than $-d\epsilon n/4$, their concatenation has
weight at least $dn(\mu-\epsilon)$. It follows from
\eqref{eq:many-tube-lower-tail} and \eqref{eq:connector-lower-tail} that
\[
 \PP\bigl(L_n<dn(\mu-\epsilon)\bigr)
 \leq \exp\{-c_5n^{1+q(d-1)}\}
      +\exp\{-c_7n^{\nu(1-q)}\}.
\]
Both exponents are strictly larger than one. The right-hand side is therefore less than $e^{-An}$ for every fixed $A>1$ once $n$ is sufficiently large.
\end{proof}

\section{Probability that the corner path is a geodesic}
\label{app:cornerpathprob}

\begin{proposition}
Assume that the weight distribution has a tiltable time constant. For every $\delta>0$ and all sufficiently large $n$,
\begin{equation}
\PP(L(\cornerpath)= L_n)\ge \exp\bigl(-dn(I(\mu)+\delta)\bigr).
\end{equation}
\end{proposition}
We use the notation $\ZZ_{[a,b]}^d := [a,b]^d \cap \ZZ^d$.
\begin{proof}
Because the time constant $\mu$ is tiltable, there exists a tilted measure under which the mean is strictly greater than $\mu$. Consequently, the essential supremum of the distribution is strictly greater than $\mu$ (possibly infinite). We may therefore choose real parameters $V_1, V_2$ and $\eta > 0$ such that $\mu<V_1<V_2$, and the events $\{\omega_x \in [V_1, V_2]\}$ and $\{\omega_x \le V_1 - \eta\}$ both have strictly positive probability.

We take $\eps>0$ small enough to ensure $\mu + 2\epsilon < V_1$. For a directed path $\pi$ from $a$ to $b$, write
\[
    \pi^\circ=\pi\setminus\{a,b\},
    \qquad
    L^\circ(\pi)=\sum_{x\in\pi^\circ}\omega_x.
\]
For $a,b\in\cornerpath$ with $a\le b$, let $\cornerpath[a,b]$ be the segment of $\cornerpath$ from $a$ to $b$, and define
\[
    L_{\cornerpath}^\circ(a,b)=L^\circ(\cornerpath[a,b]),
    \qquad
    L_{\mathrm{off}}^\circ(a,b)
    =\max_{\substack{\pi\in\Gamma(a;b)\\ \pi^\circ\cap\cornerpath=\varnothing}}L^\circ(\pi).
\]
The maximum is understood to be $-\infty$ when there is no admissible path. Thus the two passage times omit the same endpoint weights, and $L_{\mathrm{off}}^\circ(a,b)$ depends only on weights outside $\cornerpath$. For $r>0$, set
\[
    \mathcal N_r=\left\{x\in\ZZ_{[0,n]}^d:
    \exists i\in\{1,\ldots,d\},\ \left\|x-n\sum_{j=i}^d e_j\right\|_\infty\le r\right\}.
\]
We consider the events
\begin{align*}
    \mathcal E_1&:= \left\{\begin{array}{l}
    \forall a,b\in \cornerpath\text{ with $a\le b$ and $\norm{b - a}_\infty > (\log n)^2$},\\
    \hspace{2em}L_{\cornerpath}^\circ(a,b) \geq (\mu+2\eps)\norm{b-a}_1,\\
         \forall x\in\cornerpath\cap\mathcal N_{(\log n)^3},\,\omega_x\in[V_1,V_2]
    \end{array}\right\},\\
         \mathcal E_2&:= \left\{
    \begin{array}{l}
    \forall a,b\in \cornerpath\text{ with $a\le b$},\,
    L_{\mathrm{off}}^\circ(a,b) < (\mu+\eps)\norm{b-a}_1 + (\log n)^2,\\
       \text{and}\quad \forall x\in\mathcal N_{(\log n)^3}\setminus\cornerpath,\,\omega_x\le V_1-\eta
    \end{array}
    \right\},
\end{align*}
We choose $n$ so large that
$$
\epsilon (\log n)^3 > 30d\bigl((\log n)^2+\mu+1\bigr).
$$
We first prove that $\mathcal E_1\cap\mathcal E_2\subset\{L(\cornerpath)=L_n\}$.

Let $\gamma\ne \cornerpath$ be any directed path from $0$ to $(n,\ldots,n)$. Decompose $\gamma$ into maximal excursions $\pi_1,\ldots,\pi_\ell$ whose endpoints $a_j,b_j$ lie in $\cornerpath$ and whose interiors are disjoint from $\cornerpath$. It is enough to show that every excursion has smaller interior weight than the corresponding segment of $\cornerpath$.

Fix an excursion $\pi_j$ and write $a=a_j$ and $b=b_j$. Suppose first that
$$
\norm{b - a}_\infty >(\log n)^3.
$$
Then, by $\mathcal E_2$, the interior of the excursion satisfies
$$
L^\circ(\pi_j) \le L_{\mathrm{off}}^\circ(a,b) < (\mu+\epsilon)\norm{b-a}_1 + (\log n)^2.
$$
Also, since $\norm{b - a}_\infty > (\log n)^3 > (\log n)^2$, the definition of $\mathcal E_1$ gives
$$
L_{\cornerpath}^\circ(a,b) \ge (\mu+2\epsilon)\norm{b-a}_1.
$$
Since $\norm{b -a}_1 \ge (\log n)^3$, we have 
$$
(\mu+2\epsilon)\norm{b - a}_1 > (\mu+\epsilon)\norm{b - a}_1+(\log n)^2.
$$
Thus $L^\circ(\pi_j)<L_{\cornerpath}^\circ(a,b)$.

Now suppose that
$$
\norm{b - a}_\infty \le (\log n)^3.
$$
By the geometry of the corner path, $a$ and $b$ lie on two consecutive sides of $\cornerpath$ and are within distance $(\log n)^3$ of their common corner. Hence every interior vertex of $\pi_j$ lies within distance $(\log n)^3$ of that corner, and every vertex of $\cornerpath[a,b]$ lies within distance $(\log n)^3$ of it.

Therefore, under $\mathcal E_1$, every interior vertex of $\cornerpath[a,b]$ has weight at least $V_1$. Under $\mathcal E_2$, every interior vertex of $\pi_j$ has weight at most $V_1-\eta$. A genuine excursion has $\norm{b-a}_1\ge2$, and both paths have exactly $\norm{b-a}_1-1$ interior vertices. Thus,
$$
L^\circ(\pi_j)
\le (\norm{b-a}_1-1)(V_1-\eta)
< (\norm{b-a}_1-1)V_1
\le L_{\cornerpath}^\circ(a,b).
$$
Replacing each excursion of $\gamma$ by the corresponding subpath of $\cornerpath$ strictly increases the total weight, since the endpoints are shared and their weights cancel. Since $\gamma\ne \cornerpath$ was arbitrary, we obtain $\mathcal E_1\cap\mathcal E_2\subset\{L(\cornerpath)=L_n\}$.

It remains to estimate $\PP(\mathcal E_1\cap\mathcal E_2)$. By the tiltability assumption, we may choose $\epsilon>0$ so small that there is a $t_\epsilon>0$ in the interior of the effective domain of $\Lambda$ satisfying
\[
    \Lambda'(t_\epsilon)=\mu+3\epsilon,
    \qquad
    I(\mu+3\epsilon)+t_\epsilon\epsilon \le I(\mu)+\delta/3.
\]
Let $\mathbb Q_\epsilon$ be the product measure on the weights of $\cornerpath$ obtained by tilting each weight by $t_\epsilon$. Under $\mathbb Q_\epsilon$, the weights have mean $\mu+3\epsilon$. Define
\[
    G_n=\{\omega_x\in[V_1,V_2]\text{ for every }
    x\in\cornerpath\cap\mathcal N_{(\log n)^3}\},
\]
and let $K_n=\abs{\cornerpath\cap\mathcal N_{(\log n)^3}}=O((\log n)^3)$. If
$\mathbb Q_\epsilon(\omega_x\in[V_1,V_2])>0$, then
\[
    \mathbb Q_\epsilon(G_n)=\mathbb Q_\epsilon(\omega_x\in[V_1,V_2])^{K_n}=\exp\{-o(n)\}.
\]

Put $\rho=\min\{\epsilon,V_1-(\mu+2\epsilon)\}>0$. Conditional on $G_n$, the
weights on $\cornerpath$ remain independent. Every unconstrained weight has mean
$\mu+3\epsilon$, while every constrained weight is supported on $[V_1,V_2]$.
Thus every conditional mean is at least $\mu+2\epsilon+\rho$, and the two
conditional laws have a common finite negative exponential moment. For an
interval with $m$ edges, $L_{\cornerpath}^\circ$ contains $m-1$ weights. A
Chernoff bound, with the omitted endpoint absorbed by the gap of order $\rho m$,
therefore gives a constant $c>0$ such that, uniformly for $m>(\log n)^2$,
\[
    \mathbb Q_\epsilon\left(
    L_{\cornerpath}^\circ(a,b)<(\mu+2\epsilon)m\mid G_n\right)
    \le e^{-cm}.
\]
A union bound over the $O(n^2)$ intervals shows that the first condition in
$\mathcal E_1$ holds with conditional probability $1-o(1)$.

Moreover, the conditional mean of $L({\cornerpath})$ is at most
$(dn+1)(\mu+3\epsilon)+O(K_n)$. The two conditional laws also have a common
finite positive exponential moment, so another Chernoff bound gives
\[
    \mathbb Q_\epsilon\left(
    L({\cornerpath})>(dn+1)(\mu+4\epsilon)\mid G_n\right)
    \le e^{-cn}.
\]
Consequently,
\[
    \mathbb Q_\epsilon\left(
    \mathcal E_1\cap\{L({\cornerpath})\le(dn+1)(\mu+4\epsilon)\}\right)
    \ge \exp\{-o(n)\}.
\]
On this event, the Radon--Nikodym derivative of the original measure with
respect to $\mathbb Q_\epsilon$ is at least
\[
    \exp\left\{-(dn+1)
    [I(\mu+3\epsilon)+t_\epsilon\epsilon]\right\}.
\]
It follows from the choice of $\epsilon$ that, for all sufficiently large $n$,
\[
    \PP(\mathcal E_1)\ge \exp\{-d n(I(\mu)+\delta/2)\}.
\]

For $\mathcal E_2$, write
\[
    L^\circ(x,y)=\max_{\pi\in\Gamma(x;y)}L^\circ(\pi).
\]
Then
$L_{\mathrm{off}}^\circ(a,b)\le L^\circ(a,b)$. 
By an upper-tail large deviation bound for last-passage times, for any sufficiently small fixed $t>0$,
\[
    \begin{aligned}
    &\PP\left(L^\circ(a,b)\ge
    (\mu+\epsilon)\norm{b-a}_1+(\log n)^2\right) \leq \exp{(-c(\log n)^2)},
    \end{aligned}
\]
with some $c>0$. Thus a union bound over all $O(n^2)$ pairs gives the first condition in
$\mathcal E_2$ with probability $1-o(1)$. The second condition restricts only
$O((\log n)^{3d})$ sites. Because $\PP(\omega_x \le V_1 - \eta) > 0$, this constraint has probability
$$
\PP(\omega_x \le V_1 - \eta)^{O((\log n)^{3d})} = \exp\{-o(n)\}.
$$
Both conditions are decreasing events in the off-corner weights, so the Harris inequality gives $\PP(\mathcal E_2)\ge \exp\{-o(n)\}$. Finally, because $\mathcal E_1$ and $\mathcal E_2$  are independent, we conclude
$$
\PP(\mathcal E_1\cap\mathcal E_2) = \PP(\mathcal E_1)\PP(\mathcal E_2) \ge \exp\{-d n(I(\mu)+\delta)\}.
$$
\end{proof}

\begin{proposition}
Suppose the weight distribution satisfies Assumptions~\ref{ass:lower-tail} and \ref{ass:tiltable}. For every $\delta > 0$ and all sufficiently large $n$,
\begin{equation}\label{eq:lower-gamma2}
\PP(L(\cornerpath)= L_n)\le \exp\bigl(-dn(I(\mu)-\delta)\bigr).
\end{equation}
\end{proposition}
\begin{proof}
If $I(\mu)\le\delta$, then the right-hand side of \eqref{eq:lower-gamma2} is at least one and the result is immediate. We may therefore assume that $I(\mu)>\delta$. Since the tiltable distribution is nondegenerate and $I$ vanishes only at $\EE\omega_x$, it follows that $\mu>\EE\omega_x$. By the left-continuity of $I$ at $\mu$, we can choose $\epsilon>0$ such that 
\[
    \epsilon<\mu-\EE\omega_x,
    \qquad
    I(\mu-\epsilon)\ge I(\mu)-\delta/4.
\]

We define
\[
    \widehat L_n=L(e_1,n\1-e_d)+\omega_0+\omega_{n\1}.
\]
An optimizing path for the bulk last-passage time can be extended by the two connector edges $0\to e_1$ and $n\1-e_d\to n\1$, so deterministically $L_n\ge\widehat L_n$. Set
\[
    A_0=2d\bigl(I(\mu)+1\bigr),
    \qquad z_n=e_1+(n-1)\1.
\]
Fix a directed path $\sigma_n$ from $z_n$ to $n\1-e_d$ and set
\[
    R_n=\sum_{x\in\sigma_n\setminus\{z_n\}}\omega_x.
\]
 Concatenating
$\sigma_n$ with an optimizing path to $z_n$ gives deterministically
\[
    L(e_1,n\1-e_d)\ge L(e_1,z_n)+R_n.
\]
Moreover, for all sufficiently large $n$,
\[
    d(n-1)(\mu-\epsilon/4)-dn\epsilon/8
    \ge dn(\mu-\epsilon/2).
\]
 Lemma~\ref{lemma:lower-tail}, applied with $A=A_0$ and $\epsilon/4$ at scale $n-1$, therefore gives
\begin{align*}
    &\PP\bigl(L(e_1,n\1-e_d)<dn(\mu-\epsilon/2)\bigr)\\
    &\quad\le \PP\bigl(L(e_1,z_n)<d(n-1)(\mu-\epsilon/4)\bigr)
       +\PP(R_n<-dn\epsilon/8)\\
    &\quad\le e^{-A_0(n-1)}+Ce^{-cn^\nu}.
\end{align*}
Assumption~\ref{ass:lower-tail} and a union bound also give
\[
    \PP\bigl(\omega_0+\omega_{n\1}<-dn\epsilon/2\bigr)
    \le 2\PP\bigl(\omega_0<-dn\epsilon/4\bigr)
    \le 2e^{-c n^\nu}.
\]
Consequently, for all sufficiently large $n$,
\[
    \PP\bigl(\widehat L_n<dn(\mu-\epsilon)\bigr)
    \le \exp\{-dn(I(\mu)+1/2)\}.
\]

Moreover,
\[
    \{L(\cornerpath)=L_n\}\cap
    \set{\widehat L_n\ge dn(\mu-\epsilon)}
    \subseteq \set{L(\cornerpath)\ge dn(\mu-\epsilon)}.
\]
Because $\cornerpath$ consists of a fixed collection of $dn+1$ vertices and $\mu-\epsilon>\EE\omega_x$, the standard large deviation upper bound \cite{DZ:ldp} gives
\[
    \mathbb P\left(L(\cornerpath) \ge dn(\mu-\epsilon)\right) \le \exp\bigl(-dn I(\mu-\epsilon) + o(n)\bigr).
\]
Our choice of $\epsilon$ allows us to absorb the $o(n)$ term and obtain, for all sufficiently large $n$,
\[
    \mathbb P\left(L(\cornerpath) \ge dn(\mu-\epsilon)\right)
    \le \exp\{-dn(I(\mu)-\delta/2)\}.
\]
Combining these bounds via a union bound yields
\begin{align*}
    \PP(L(\cornerpath)= L_n) &\leq \PP\bigl(\widehat L_n < dn(\mu-\epsilon) \bigr) + \PP(L(\cornerpath) \ge dn(\mu-\epsilon)) \\
    &\le \exp\{-dn(I(\mu)+1/2)\}+\exp\{-dn(I(\mu)-\delta/2)\}\\
    &\le \exp\{-dn(I(\mu)-\delta)\},
\end{align*}
where the last inequality holds for sufficiently large $n$ because $dn\delta/2\ge\log 2$.
\end{proof}

\section{Sharp lower-tail moderate deviations under exponential weights}
\label{app:sharpmd}
Recall that we know exactly what the moderate deviations of the point-to-point time $L_n$ look like for integrable weights. 
\begin{theorem}[\cite{BDMMZ:tail}]
\label{thm:sharpp2pmd}
Suppose $d = 2$ and $\omega_x \sim \Exp(1)$. With $T > 0$ a large constant and $\rho > 0$ a small constant, the following holds uniformly for $T \le t \le \rho n^{2/3}$:
\begin{equation}
\label{eq:lefttail-app}
    \log \PP(L_n \le 2 \mu n -  t n^{1/3}) = -\frac{1} {192} t^{3} + O(t^4 n^{-2/3}) + O(\log t).
\end{equation}
\end{theorem}
Note that this is proved in \cite{BDMMZ:tail} for geometric weights. The method they use is also applicable to exponential weights, and the argument is sketched in \cite{L:iteratedlog}.

In the course of proving Theorem~\ref{thm:probability} and Corollary~\ref{cor:staircases}, it is useful to have similar control over the point-to-line and half-space times, respectively. Various lower-tail bounds appear in the literature, usually in the context of the $\beta$-Laguerre ensemble. Our point-to-line, point-to-point, and half-space times are contained in this model by taking the special values $\beta = 1,\, 2,\, 4$, respectively. An upper bound on the lower-tail probabilities with the correct cubic order was found in \cite{LR:betadeviations}, while the complementary lower bound was obtained in \cite{BGHK:lpp-lis}. Upper and lower bounds with matching leading coefficients were established in the comprehensive results \cite{BBBK:tail-bounds}, although the applicable range of deviations and the control of the error terms are unfortunately insufficient for our use case. We therefore prove improved statements for the special cases we care about.

Write $L^{\bltriangle}_n = L^{\bltriangle}(0, 0; n)$ and $L^{\brtriangle}_n = L^{\brtriangle}(0, 0; n, n)$.

\begin{theorem}
\label{thm:sharpmd}
Suppose $d = 2$ and $\omega_x \sim \Exp(1)$. With $T > 0$ a large constant and $\rho > 0$ a small constant, the following hold uniformly for $T \le t \le \rho n^{2/3}$:
\begin{equation}
 \log \PP \brac*{L^{\bltriangle}_n \le \mu n - t n^{1/3}}
 = -\frac{1}{48} t^3 + O(t^4 n^{-2/3}) + O(t^{3/2})
 \label{eq:sharpp2l}
\end{equation}
and
\begin{equation}
 \log \PP \brac*{L^{\brtriangle}_n \le 2 \mu n- t n^{1/3}}
 = -\frac{1}{384} t^3 + O(t^4 n^{-2/3}) + O(t^{3/2}).
 \label{eq:sharphalf}
\end{equation}
\end{theorem}

These are precisely the coefficients we would expect had we equalities rather than inequalities in \eqref{eq:passagetimerelations}, in light of Theorem~\ref{thm:sharpp2pmd}. Indeed, two of the four inequalities we need to prove follow from this geometric relationship between the passage times. The other directions are less direct, and we appeal to a certain miraculous coupling between the classical Laguerre ensembles found in \cite{FR:interrelationships}. We should expect neither the geometric nor the coupling approach to be useful for general $\beta \ge 1$. This proof of the hard directions was formulated by OpenAI's GPT-5.6 Sol.

We begin with the easy directions.
\begin{lemma}\label{lem:easysharpmd}
Uniformly for $T \le t \le \rho n^{2/3}$,
\[
    \log\PP \brac*{L^{\bltriangle}_n \le \mu n - t n^{1/3}} \le -\frac{1}{48} t^3 + O(t^4 n^{-2/3}) + O(\log t)
\]
and 
\[
    \log \PP \brac*{L^{\brtriangle}_n \le 2 \mu n- t n^{1/3}} \ge -\frac{1}{384} t^3 + O(t^4 n^{-2/3}) + O(\log t).
\]
\end{lemma}

\begin{proof}
Recall from \eqref{eq:passagetimerelations} that 
\[
L_n \le L^{\bltriangle}(0, 0; n) + L^{\trtriangle}(n + 1; n, n).
\]
The first term is $L^{\bltriangle}_n$, and by symmetry the latter term is equal in distribution to $L^{\bltriangle}_{n - 1}$. So certainly $L_n$ is stochastically dominated by the sum of two independent copies of $L^{\bltriangle}_n$. Write $\tilde{L}^{\bltriangle}_n$ for such a copy. Then 
\begin{align*}
    \PP(L_n \le 2 \mu n - 2 t n^{1/3}) &\ge \PP(L^{\bltriangle}_n \le \mu n - t n^{1/3},\, \tilde{L}^{\bltriangle}_n \le \mu n - t n^{1/3})\\ 
    &= \PP(L^{\bltriangle}_n \le \mu n - t n^{1/3})^2.
\end{align*}
We have an expression for the left side by plugging into Theorem~\ref{thm:sharpp2pmd}, and taking square roots yields the first claimed inequality.

The same argument works for the second inequality. Now the relevant relation is
\[
    L_{n + 1} \ge L^{\brtriangle}(0, 0; n + 1, n + 1) \vee L^{\tltriangle}(0, 1; n, n + 1) \disteq L^{\brtriangle}_{n + 1} \vee \tilde{L}^{\brtriangle}_{n}.
\] 
Here $\tilde{L}^{\brtriangle}_{n }$ is an independent copy of $L^{\brtriangle}_{n}$. Thus $L_{n + 1}$ stochastically dominates $L^{\brtriangle}_{n} \vee \tilde{L}^{\brtriangle}_{n}$, and 
\begin{align*}
\PP(L_{n + 1} \le 2 \mu n - t n^{1/3}) &\le \PP(L^{\brtriangle}_{n} \le 2 \mu n - t n^{1/3},\, \tilde{L}^{\brtriangle}_{n} \le 2 \mu n - t n^{1/3}) \\
&= \PP(L^{\brtriangle}_n \le 2 \mu n - t n^{1/3})^2.
\end{align*}
Once again, taking the square root of the expression in Theorem~\ref{thm:sharpp2pmd} yields the result.
\end{proof}

Before proceeding with the reverse directions, we must introduce the classical Laguerre ensembles. Given an integer $N \ge 1$, these are collections of $N$ ordered eigenvalues $x_0 < x_1 < \cdots < x_{N - 1}$ with the distributions below.
\begin{align*}
 \LOE_N &\sim\frac{1}{Z} \prod_{i < j} \abs{x_i - x_j} \prod_i e^{-x_i / 2},\\
 \LUE_N &\sim\frac{1}{Z} \prod_{i < j} \abs{x_i - x_j}^2 \prod_i e^{-x_i},\\
 \LSE_{M} &\sim\frac{1}{Z} \prod_{i < j} \abs{x_i - x_j}^4 \prod_i e^{-x_i}.
\end{align*}
For the last definition, $N = 2 M$ is even and there are $M$ eigenvalues. The normalising constant $Z$ differs for each. We have defined only the specialisations which show up as passage times in our LPP model  -- typically one has a shape parameter and the rates are chosen differently. With our definitions, the following distributional identities between passage times and maximal eigenvalues hold:
\begin{equation}
\label{eq:laguerretolpp}
\begin{aligned}
    L_n &\disteq \lambda_{\max}(\LUE_{n + 1}),\\
    L^{\bltriangle}_n &\disteq \frac{1}{2} \lambda_{\max}(\LOE_{n + 1}),\\
    L^{\brtriangle}_n &\disteq \lambda_{\max}(\LSE_{n / 2 + 1}).
\end{aligned}
\end{equation}
These are the identities listed in \cite[(1.10)]{BBBK:tail-bounds} translated into our notation. The half-space identity is valid only for $n$ even, but we will allow ourselves to be loose with this.

Given an ordered sequence $X = (x_0, x_1, \dots)$, define $\operatorname{even}(X) = (x_0, x_2, \dots)$. That is, we retain the smallest entry, third smallest entry, and so on.
\begin{theorem}[{\cite[Theorem~5.1]{FR:interrelationships}}]
    \label{thm:FRidentities}
    Sample independent sequences $X, \tilde{X} \sim \LOE_N$. Then 
    \begin{equation}
    \label{eq:LOEtoLUE}
    \operatorname{even}(X \cup \tilde{X}) \sim \LUE_N
    \end{equation}
    and if $N = 2 M$ is even, also 
    \begin{equation}
    \label{eq:LOEtoLSE}
    \operatorname{even}(X) \sim \LSE_M.
    \end{equation}
\end{theorem}

We have control of $\lambda_{\max} (\LUE)$ through Theorem~\ref{thm:sharpp2pmd}. In light of \eqref{eq:laguerretolpp}, we should try to transfer this control to $\lambda_{\max} (\LOE)$ and then to $\lambda_{\max} (\LSE)$. Take $X \sim \LOE_N$ and for $j \ge 0$ define 
\begin{equation*}
    p_j(x) = \PP\brac[\Big]{\abs{X \cap (x, \infty)} = j}.
\end{equation*}
In words, $p_j(x)$ is the probability that exactly $j$ eigenvalues of $X$ exceed $x$. When $X' \sim \LUE_N$, define the corresponding probabilities
\begin{equation*}
    q_j(x) = \PP\brac[\Big]{\abs{X' \cap (x, \infty)} = j}.
\end{equation*}

\begin{corollary}
\label{cor:pjrelations}
With the notation above,
\begin{equation}
\label{eq:LUEprob1}
\PP\brac[\big]{\lambda_{\max} (\LUE_N) \le x} = q_0(x) = p_0(x)^2 + 2 p_0(x) p_1(x)
\end{equation}
and 
\begin{equation}
\label{eq:LUEprob2}
q_1(x) \ge p_1(x)^2.
\end{equation}
If $N = 2M$, then also
\begin{equation}
\label{eq:LSEprob}
\PP\brac[\big]{\lambda_{\max} (\LSE_M) \le x} = p_0(x) + p_1(x).
\end{equation}
\end{corollary}
These expressions are immediate from Theorem~\ref{thm:FRidentities}. If $p_0(x)$ and $p_1(x)$ live on the same approximate scale, then asymptotics for the right quantity in \eqref{eq:LUEprob1} can be transferred to $p_0(x)$ and afterwards to the right side of \eqref{eq:LSEprob}. Showing this is the main technical difficulty.
\begin{lemma}
\label{lem:p0p1ratio}
If $1 \le s \le cN^{2 / 3}$ and $x = 4N - s N^{1/3}$, then
\begin{equation}
    \label{eq:p0p1ratio}
    \log\brac*{1 + \frac{2 p_1(x)}{p_0(x)}} = O(s^{3/2})
\end{equation}
uniformly on this range.
\end{lemma}
\begin{proof}
Let $P = p_1 / p_0$ and $Q = q_1 / q_0$. By \eqref{eq:LUEprob1} and \eqref{eq:LUEprob2}, we have
\[
Q \ge \frac{P^2}{1 + 2 P},
\]
which implies
\[
1 + 2 P \le 4(1 + Q).
\]
We are trying to bound $\log(1 + 2 P)$, hence it suffices to bound $\log(1 + Q)$. This is more appealing, as $Q$ comes from the $\LUE_N$ ensemble and admits a determinantal formula. The ensemble is associated with the finite dimensional Hilbert space of weighted polynomials
\[
    \mathcal{H}_N = \set{e^{-y / 2} R(y) : \deg R < N} \subset L^2(0, \infty),
\]
The projection $K^{(N)} : L^2(0, \infty) \to \mathcal{H}_N$ is the kernel encoding the joint intensities of the ensemble. See Sections 4.1 and 4.2 of \cite{AGZ:randommatrices} for background on determinantal point processes and the special case of the Laguerre ensemble. 

Further define $M_x$ as multiplication by $\bone_{(x, \infty)}$, and 
\[
K^{(N)}_x = M_x K M_x.
\]
Let $\nu_0, \dots, \nu_{N - 1}$ be the eigenvalues of $K^{(N)}_x$. Corollary 4.2.24 of \cite{AGZ:randommatrices} yields the formulas
\[
q_0(x) = \prod_{i = 0}^{N - 1}(1 - \nu_i),\qquad q_1(x) = \sum_{i = 0}^{N - 1} \sqrbrac[\bigg]{\nu_i \prod_{j \ne i}(1 - \nu_j)}.
\]
Hence
\[
Q = \sum_{i = 0}^{N - 1} \frac{\nu_i}{1 - \nu_i}.
\]
Set $A = \norm{K^{(N)}_x}$. Then $1 - \nu_i \ge 1 - A$ for all $i$, and 
\[
Q \le (1 - A)^{-1}\sum_{i = 0}^{N - 1} \nu_i \le - (1 - A)^{-1} \log q_0(x).
\]
Here we used $\log q_0(x) = \sum_i \log (1 - \nu_i)$ and the fact that $\log (1 + y) \le y$. Now, we know from Theorem~\ref{thm:sharpp2pmd} that $- \log q_0(x) = \frac{1}{192} s^3 (1 + o(1))$, and we will show in Lemma~\ref{lem:polynomialbound} that $1 - A \ge e^{-C s^{3/2}}$, with $C > 0$ uniform on the range of $s$. Combining these gives the desired 
\[
\log(1 + 2P) \le \log(1 + Q) + \log 4 \le C s^{3 / 2} + O(\log s).
\]
\end{proof}

The bound on the eigenvalues used to estimate the ratio $q_1 / q_0$ is in fact an extremely concrete bound concerning the concentration of polynomials.
\begin{lemma}
\label{lem:polynomialbound}
There exist constants $\rho, C > 0$ such that for $1 \le s \le \rho N^{2/3}$ and $x = 4N - s N^{1/3}$, every polynomial $R$ of degree at most $N - 1$ satisfies
\begin{equation}
\label{eq:polynomialbound}
    \int_0^\infty \abs{R(y)}^2 e^{-y} \, dy \le e^{C s^{3/2}} \int_0^x \abs{R(y)}^2 e^{-y} \, dy.
\end{equation}
In particular, if $K^{(N)}_x$ is the kernel defined above, one has (possibly after adjusting the constant $C > 0$) that
\[
\norm{K^{(N)}_x} \le 1 - e^{-C s^{3/2}}.
\]
\end{lemma}

\begin{proof}
The estimate will be obtained by standard techniques from complex analysis after transplanting the problem to a setting where these techniques are especially powerful, namely to the Hardy space $H^2$ of analytic functions on the unit disc. To that end, consider the Joukowski map
\[
    X(z) = \frac{x}{4} \brac{2 + z + z^{-1}}.
\]
This has the useful property of mapping the unit circle onto $[0, x]$ and the radius $(0, 1)$ onto $(x, \infty)$. More explicitly,
\[
    X(e^{i\theta}) = \frac{x}{2} \brac{1 + \cos \theta},
\]
and 
\[
    X(e^{-v}) = \frac{x}{2} \brac{1 + \cosh v}.
\]
Composing the polynomial $R$ with the Joukowski map will bring it to the disc in a sensible way, but we need an additional algebraic trick. Define
\[
    F(z) = z^{N - 1} R(X(z)) e^{-x (1 + z) / 4}.
\]
The $z^{N - 1}$ factor cancels out any negative powers appearing in $R(X(z))$, hence it is an entire function. Observe that 
\[
    \abs{F(e^{i\theta})}^2
    = \abs{R(X(e^{i\theta}))}^2 e^{-X(e^{i\theta})}
\]
The change of variables $y = X(e^{i \theta})$ gives $dy = - \frac{x}{2} \sin \theta \, d\theta$. Hence if we define 
 \[
    G(z) = F(z) \brac{1 - z^2}^{1/2},
 \]
under the change of variables we get
\[
    \int_0^x \abs{R(y)}^2 e^{-y} \, dy
    = \frac{x}{8} \int_0^{2 \pi} \abs{G(e^{i \theta})}^2 \, d\theta.
\]
This $G$ is analytic in the disc and continuous on its closure, and in particular $G \in H^2$. For concreteness, define the norm on the space by
\[
    \norm{G}_{H^2}^2
    = \frac{1}{2\pi} \int_0^{2\pi} \abs{G(e^{i\theta})}^2 \, d\theta.
\]
Thus we end up with
\begin{equation}
    \label{eq:H2inside}
    \int_0^x \abs{R(y)}^2 e^{-y} \, dy = \frac{\pi x}{4} \norm{G}_{H^2}^2.
\end{equation}

This is a workable expression for the integral over $[0, x]$. Next we obtain a suitable bound on the integral over $(x, \infty)$. For $y = X(e^{-v})$, the definition of $F$ gives
\[
    \abs{R(y)}^2 e^{-y} = \abs{F(e^{-v})}^2 \exp\brac*{2(N - 1)v - \frac{x}{2} \sinh v}.
\]
A change of variables leads to
\[
    \int_x^\infty \abs{R(y)}^2 e^{-y} \, dy = \frac{x}{4} \int_0^\infty \abs{G(e^{-v})}^2 \exp\brac*{(2N - 1)v - \frac{x}{2} \sinh v} \, dv,
\]
while introducing $r = e^{-v}$ brings this to
\begin{equation}
    \label{eq:H2outside}
    \int_x^\infty \abs{R(y)}^2 e^{-y} \, dy = \frac{x}{4} \int_0^1 \abs{G(r)}^2 e^{\Phi_x(-\!\log r)} \, dr.
\end{equation}
Here
\[
    \Phi_x(v) = 2Nv - \frac{x}{2} \sinh v.
\]
The fact we appeal to now is the following instance of the Fej\'er--Riesz inequality \cite[Theorem~3.13]{D:hardyspaces}, linking the integral of $G$ along the radius to the integral along the circle. For all $h \in H^2$, we have
\begin{equation*}
    \int_0^1 \abs{h(r)}^2 \, dr \le \pi \norm{h}_{H^2}^2.
\end{equation*}
Thus from \eqref{eq:H2outside}, there is the bound 
\begin{equation}
\label{eq:H2outside2}
    \int_x^\infty \abs{R(y)}^2 e^{-y} \, dy \le \sqrbrac[\bigg]{\sup_{v \ge 0} \exp\brac[\big]{\Phi_x(v)}} \frac{\pi x}{4}  \norm{G}_{H^2}^2.
\end{equation}
It remains to analyse $\Phi_x(v)$. This is an exercise in calculus, and one finds for $s$ not too big (and hence the constant $\rho$ chosen small enough) that
\[
\sup_{v \ge 0} \Phi_x(v) \le \frac{\sqrt{2}}{3 \sqrt{3}} s^{3/2}.
\]
Write $C$ for the constant in the exponent. Combining this with the ingredients of \eqref{eq:H2inside} and \eqref{eq:H2outside2}, we finally get 
\[
    \int_x^\infty \abs{R(y)}^2 e^{-y} \, dy \le e^{C s^{3/2}}\frac{\pi x}{4}  \norm{G}_{H^2}^2 = e^{C s^{3/2}} \int_0^x \abs{R(y)}^2 e^{-y} \, dy.
\]
Some rearrangement brings this into the form of the desired statement.
\end{proof}

\begin{proof}[Proof of Theorem~\ref{thm:sharpmd}]
Half of each expression is proved in Lemma~\ref{lem:easysharpmd}, so it remains to show 
\[
 \log\PP \brac*{L^{\bltriangle}_n \le \mu n - t n^{1/3}} \ge -\frac{1}{48} t^3 + O(t^4 n^{-2/3}) + O(t^{3/2})
 \]
and
\[
 \log \PP \brac*{L^{\brtriangle}_n \le 2 \mu n- t n^{1/3}} \le -\frac{1}{384} t^3 + O(t^4 n^{-2/3}) + O(t^{3/2}).
\]
Consulting \eqref{eq:laguerretolpp}, the first event is equivalent to 
\[
\lambda_{\max}(\LOE_{n + 1}) \le 4 n - 2 t n^{1/ 3}.
\]
Write $x = 4 n - 2 t n^{1/ 3}$. The quantity on the left is exactly $p_0(x)$. Recalling \eqref{eq:LUEprob1}, we have the exact expression
\[
2 \log p_0(x) = \log q_0(x) - \log\brac[\bigg]{1 + \frac{2 p_1(x)}{p_0(x)}}.
\]
Theorem~\ref{thm:sharpp2pmd} tells us the asymptotic value of $q_0(x)$, and the second term is $O(t^{3/2})$ by Lemma~\ref{lem:p0p1ratio}. Hence 
\[
\log p_0(x) \ge -\frac{1}{48} t^3 + O(t^4 n^{-2/3}) - O(t^{3/2}).
\]
The second event can be treated in the same way, using \eqref{eq:LSEprob} as the starting point.
\end{proof}

\section{Moderate deviations of transversal fluctuations}
\label{app:transflucts}
One can use the moderate deviations of passage times to bound the probability of atypically large transversal fluctuations of the corresponding geodesics. This has been observed, for example, in \cite{J:transflucts,BB:smalltransdeviations}. 

The bounds we need to prove Corollary~\ref{cor:staircases} do not seem to be stated explicitly in the literature, so we give the argument here. Showing that the obtained bounds are of optimal order is another question, which thankfully is not important for us.

Given a path $\gamma \in \Gamma_n$, let $Q(\gamma) = \max \set{\abs{x - y} : (x, y) \in \gamma}$ be the radius (up to a factor of $\sqrt{2}$) of the minimal tube around the straight line $(0, 0) \to (n, n)$ containing $\gamma$. Recall that $\gamma^*_n$ is the path attaining the passage time $L(0, 0; n, n)$. Define also $\gamma^{\brtriangle}_n$ to be the path attaining the half-space time $L^{\brtriangle}(0, 0; n, n)$.

\begin{proposition}
\label{prop:moderatetransflucts}
Consider $d = 2$ with $\omega_x \sim \mathrm{Exp}(1)$. Take $\alpha \in (\frac{2}{3}, 1]$. There is $c > 0$ such that
\[
\PP(Q(\gamma_n^*) \ge n^{\alpha}) \le e^{-c n^{3 \alpha - 2}}.
\]
The same is true of $Q(\gamma^{\brtriangle}_n)$.
\begin{proof}
    For simplicity we treat only the point-to-point geodesic $\gamma^*_n$, but the proof goes through essentially unchanged for $\gamma^{\brtriangle}_n$ (and for point-to-line geodesics, should we need them). The case $\alpha = 1$ has been treated in Appendix~\ref{app:cornerpathprob}, so assume $\alpha < 1$. We abbreviate taking integer parts. 
    
    For $t$ such that $n^{\alpha-1}\le t\le1-n^{\alpha-1}$, put $z_t=(tn+n^\alpha,tn-n^\alpha)$ and let $A_t$ be the event that $z_t\in\gamma_n^*$. The event that the geodesic exits the other side of the tube may be handled in the same way. On the event $A_t$ we have
    \[
    L(0, 0; n, n) = L(0, 0; z_{t,1},z_{t,2}) + L(z_{t,1},z_{t,2}; n, n)-\omega_{z_t}.
    \]
    Let $L_0(a, b; c, d) = L(a, b; c, d) - \mu(c - a, d - b)$ be centred passage times. A second-order Taylor expansion of the shape function in the transverse direction gives
    \[
        \mu(n, n) - \mu(t n + n^{\alpha}, t n - n^{\alpha}) - \mu\bigl((1 - t) n - n^{\alpha}, (1 - t) n + n^{\alpha}\bigr) \ge \frac{1}{2t(1 - t)}n^{2\alpha - 1},
    \]
    for $n$ large enough. Since $\omega_{z_t}\ge0$, dropping the resulting nonpositive term gives the inequality
    \begin{multline}
    \label{eq:transfluctsineq}
        L_0(0, 0; n, n) \le L_0(0, 0; t n + n^{\alpha}, t n - n^{\alpha}) \\ + L_0(t n + n^{\alpha}, t n - n^{\alpha}; n, n) -\frac{1}{2t(1 - t)}n^{2\alpha - 1}.
    \end{multline}
    Set
    \begin{align*}
        L_0^1 &= L_0(0, 0; n, n)\\
        L_0^2 &= L_0(0, 0; t n + n^{\alpha}, t n - n^{\alpha})\\
        L_0^3 &= L_0(t n + n^{\alpha}, t n - n^{\alpha}; n, n).
    \end{align*}
    Observe that $L_0^3 \disteq L_0(0, 0; (1 - t) n + n^{\alpha}, (1 - t) n - n^{\alpha})$. As in the proof of the upper bound for Theorem~\ref{thm:probability}, we apply a union bound to the event in \eqref{eq:transfluctsineq}. For each summand, the Harris inequality bounds the intersection of the decreasing event for $L_0^1$ and the increasing event for $L_0^2$ or $L_0^3$ by the product of their probabilities. Using also $t(1 - t) \le \frac{1}{4}$, we obtain
    \begin{equation}
        \label{eq:transfluctsunion}
        \PP(A_t) \le \sum_{a = - \infty}^{\infty} \PP(L_0^1 \le a + 1)\brac[\Big]{\PP\brac[\big]{2 L_0^2 \ge a + n^{2 \alpha - 1}} + \PP\brac[\big]{2 L_0^3 \ge a + n^{2 \alpha - 1}}}.
    \end{equation}
    Larger values of $a$ relax the condition on $L_0^1$ while forcing $L^2_0$ and $L^3_0$ further into their upper-tail, while smaller values push $L_0^1$ into its lower-tail and increase the probabilities for the remaining times. 
    
    The passage times are asymptotically Tracy-Widom GUE variables, and as such they have asymmetric tails. The estimates in \cite{LR:betadeviations} tell us that the upper-tail decays as $\exp(- c t^{3/2}/\sqrt{n})$, whereas the lower-tails have the much harsher $\exp( - c t^3 / n)$ decay. We find that the main contribution to the sum will happen at $a \sim - n^{\alpha / 2}$, and that the leading cost is on the scale $\exp(-c n^{3 \alpha - 2})$. Moreover, the contribution of $L^2_0$ dominates if $t \ge \frac{1}{2}$, and otherwise $L^3_0$ dominates. An appeal to Laplace's method shows
    \[
        \PP(A_t) \le e^{- c_t n^{3 \alpha - 2}}.
    \]
    The dependence on $t$ in the constant $c_t$ is through the scaling needed to bring the passage times to their Tracy-Widom approximations. A more careful accounting shows that it may be taken uniform in $t$.

    To complete the proof, we should bound the probability that the geodesic leaves the tube at \emph{some} $t$, rather than a particular $t$. But there are only $n$ possible choices of $t$, so a union bound  on $\bigcup A_t$ quickly gives the desired estimate.
\end{proof}
\end{proposition}

An immediate application of this bound is the fact that with very large probability, two geodesics of length $n$ with well-separated endpoints will not meet. We have only proved the fluctuation bound for diagonal geodesics, but this is all we will make use of.
\begin{corollary}
\label{cor:nonintersect}
Take $\alpha \in (\frac{2}{3}, 1]$ and fix $\theta\in(0,1]$. Consider points $u,u'\in\ZZ^2$ such that
\[
    \norm{u - u'}_1\ge n^\alpha,
    \qquad
    \bigl|\langle u - u',e_1 - e_2\rangle\bigr|
    \ge \theta\norm{u - u'}_1.
\]
Thus the displacement of the two starting points has a component transverse to the diagonal direction $e_1 + e_2$ of order at least $n^\alpha$. Take $a,a'\in\bbrac{1, n}$, and let $\gamma,\gamma'$ be the LPP geodesics joining $u\to u+a \1  $ and $u'\to u' + a' \1  $, respectively. Then there is $c>0$ such that
\[
    \PP(\gamma \cap \gamma' \ne \emptyset) \le e^{- c n^{3 \alpha - 2}}.
\]
\end{corollary}
\begin{proof}
Let $R=\frac{1}{2}|\langle u-u',e_1-e_2\rangle|$. If $\gamma$ and $\gamma'$ meet, the triangle inequality implies that $Q(\gamma - u) \ge R$ or $Q(\gamma' - u') \ge R$. Since $R \ge \frac{\theta}{2} n^\alpha$ and $a, a' \le n$, Proposition~\ref{prop:moderatetransflucts} therefore gives
\[
    \PP(\gamma\cap\gamma'\ne\emptyset)
    \le C e^{-cR^3/a^2}+C e^{-cR^3/(a')^2}
    \le e^{-c n^{3\alpha-2}}.
\]
\end{proof}

One can imagine that passage times are independent when their geodesics do not intersect, and morally this is correct. We record the precise statement in the next proposition.
\begin{proposition}
\label{prop:highprobindependence}
    Carry over the notation from Corollary~\ref{cor:nonintersect}. Write $L = L(u; u + a \1)$ and $L' = L(u'; u' + a' \1)$. Take events $\mathcal{E} \in \sigma(L)$ and $\mathcal{E}' \in \sigma(L')$. There is a constant $c > 0$, uniform in $n$ and across all choices of events $\mathcal{E},\, \mathcal{E}'$, such that
    \[
    \abs[\big]{\PP(\mathcal{E} \cap \mathcal{E}') - \PP(\mathcal{E})\PP(\mathcal{E}')} \le e^{- c n^{3 \alpha - 2}}.
    \]
    That is, the events may be considered as independent, up to an error of $\exp\brac{- c n^{3 \alpha - 2}}$.
\begin{proof}
    Write $\mathcal{T}$ for the tube of width $\lambda n^{\alpha}$ around the line joining $u \to u + a \1$, and $\mathcal{T}'$ for the tube of width $\lambda n^{\alpha}$ around the line $u' \to u' + a' \1$. Choose $\lambda > 0$ small enough that $\mathcal{T} \cap \mathcal{T}' = \emptyset$, which may be done uniformly in the angle $\theta$. Define coupled weights $\tilde{\omega}$ which are equal to $\omega$ on $\mathcal{T}$ and sampled independently elsewhere. Do the same with weights $\tilde{\omega}'$ and $\mathcal{T}'$. Observe that the weights $\tilde{\omega}$ and $\tilde{\omega}'$ are completely independent of one another.
    
    Sample $L = L(u; u + a \1)$ and $L' = L(u'; u' + a' \1)$ according to the usual weights $\omega$. Sample $M$ and $M'$ as passage times between the same points but under weights $\tilde{\omega}$ and $\tilde{\omega}'$, respectively. This choice of $M$ and $M'$ ensures they are independent by construction, and they are of course equal in distribution to their counterparts under weights $\omega$. Moreover, we will have $L = M$ so long as the geodesics defining both remain in $\mathcal{T}$, and the same for $L',\, M'$. Write $I$ for the event that none of the four geodesics leave their corresponding tube and note that $L = M$ and $L' = M'$ conditioned on $I$.
    
    Returning to the claim, choose sets $A,\, A' \subseteq \RR$ such that $\mathcal{E} = \set{L \in A}$ and $\mathcal{E}' = \set{L' \in A'}$. We can carry out the following manipulation:
    \begin{align*}
        \PP(\mathcal{E} \cap \mathcal{E}') &\le \PP(\mathcal{E} \cap \mathcal{E}' \mid I)\PP(I) + \PP(I^c)\\
                       &= \PP(L \in A,\, L' \in A' \mid I)\PP(I) + \PP(I^c)\\
                       &= \PP(M \in A,\, M' \in A' \mid I)\PP(I) + \PP(I^c)\\
                       &= \PP(M \in A,\, M' \in A',\, I) + \PP(I^c)\\
                       &\le \PP(M \in A,\, M' \in A') + \PP(I^c)\\
                       &= \PP(\mathcal{E}) \PP(\mathcal{E}') + \PP(I^c).
    \end{align*}
    We may estimate $\PP(I^c)$ using Corollary~\ref{cor:nonintersect} and a union bound, and get
    \begin{align*}
        \PP(\mathcal{E} \cap \mathcal{E}') &\le \PP(\mathcal{E}) \PP(\mathcal{E}') + 4 e^{- c n^{3 \alpha - 2}}.
    \end{align*}
    Reversing the roles of the $L$ and $M$ variables gives the inequality in the other direction.
\end{proof}
\end{proposition}

\section*{Acknowledgements}

S.N.\ acknowledges support from JSPS KAKENHI Grant Numbers JP24K16937 and JP25K00911.
\printbibliography
\end{document}

%% file: fig_bump.tex
\begin{tikzpicture}[scale=0.9, >=stealth, very thick]
    \colorlet{pathblue}{blue!50}
    
    \begin{scope}[xshift=0cm]
        \draw[step=1cm, gray!30, thin] (0,0) grid (6,6);
        \draw[dashed, gray] (0,0) -- (6,6); 

        \draw[->, pathblue] (0,0) -- (0,3) node[midway, left, text=black] {$a_1$};
        \draw[->, pathblue] (0,3) -- (2,3) node[midway, above, text=black] {$a_2$};
        \draw[->, pathblue] (2,3) -- (2,4) node[midway, left, text=black] {$a_3$};
        \draw[->, pathblue] (2,4) -- (5,4) node[midway, above, text=black] {$a_4$};
        \draw[->, pathblue] (5,4) -- (5,6) node[midway, left, text=black] {$a_5$};
        \draw[->, pathblue] (5,6) -- (6,6) node[midway, above, text=black] {$a_6$};

        \fill (0,0) circle (2pt); \fill (6,6) circle (2pt);
        \fill (0,3) circle (2pt); \fill (2,3) circle (2pt);
        \fill (2,4) circle (2pt); \fill (5,4) circle (2pt); \fill (5,6) circle (2pt);
        
        \draw[red, thick] (2,3) circle (4pt) node[below right, text=black] {$k$};
    \end{scope}

    \draw[->, ultra thick, black] (6.8, 3) -- (8.2, 3) node[midway, above] {Bump};

    \begin{scope}[xshift=9cm]
        \draw[step=1cm, gray!30, thin] (0,0) grid (6,6);
        \draw[dashed, gray] (0,0) -- (6,6); 

        \draw[->, pathblue] (0,0) -- (0,4) node[midway, left, text=black] {$a_1 + a_3$};
        \draw[->, pathblue] (0,4) -- (5,4) node[midway, above, text=black] {$a_2 + a_4$};
        \draw[->, pathblue] (5,4) -- (5,6) node[midway, left, text=black] {$a_5$};
        \draw[->, pathblue] (5,6) -- (6,6) node[midway, above, text=black] {$a_6$};

        \fill (0,0) circle (2pt); \fill (6,6) circle (2pt);
        \fill (0,4) circle (2pt); \fill (5,4) circle (2pt); \fill (5,6) circle (2pt);
    \end{scope}
\end{tikzpicture}

%% file: fig_unwind.tex
\begin{tikzpicture}[scale=0.9, >=stealth, very thick]
    \colorlet{pathblue}{blue!50}
    
    \begin{scope}[xshift=0cm]
        \draw[step=1cm, gray!30, thin] (0,0) grid (7,7);
        \draw[dashed, gray] (0,0) -- (7,7); 
    
        \draw[->, pathblue] (0,0) -- (0,2) node[midway, left, text=black] {$a_1$};
        \draw[->, pathblue] (0,2) -- (5,2) node[midway, above, text=black] {$a_2$};
        \draw[->, pathblue] (5,2) -- (5,4) node[midway, right, text=black] {$a_3$};
        \draw[->, pathblue] (5,4) -- (6,4) node[midway, below, text=black] {$a_4$};
        \draw[->, pathblue] (6,4) -- (6,7) node[midway, right, text=black] {$a_5$};
        \draw[->, pathblue] (6,7) -- (7,7) node[midway, above, text=black] {$a_6$};

        \fill (0,0) circle (2pt); \fill (7,7) circle (2pt);
        \fill (0,2) circle (2pt); 
        \fill (5,2) circle (2pt); 
        \fill (5,4) circle (2pt); 
        \fill (6,4) circle (2pt); 
        \fill (6,7) circle (2pt);
        
        \draw[red, thick] (0,2) circle (4pt) node[above left, text=black] {$l$};
        \draw[red, thick] (6,7) circle (4pt) node[above left, text=black] {$k$};
    \end{scope}

    \draw[->, ultra thick, black] (7.8, 3.5) -- (9.2, 3.5) node[midway, above] {Unwind};

    \begin{scope}[xshift=10cm]
        \draw[step=1cm, gray!30, thin] (0,0) grid (7,7);
        \draw[dashed, gray] (0,0) -- (7,7); 

        \draw[->, pathblue] (0,0) -- (0,5) node[midway, left, text=black] {$a_1 + a_5$};
        \draw[->, pathblue] (0,5) -- (1,5) node[midway, below, text=black] {$a_4$};
        \draw[->, pathblue] (1,5) -- (1,7) node[midway, right, text=black] {$a_3$};
        \draw[->, pathblue] (1,7) -- (7,7) node[midway, above, text=black] {$a_6 + a_2$};

        \fill (0,0) circle (2pt); \fill (7,7) circle (2pt);
        \fill (0,5) circle (2pt); 
        \fill (1,5) circle (2pt); 
        \fill (1,7) circle (2pt);
    \end{scope}

\end{tikzpicture}

%% file: fig_staircase.tex
\begin{tikzpicture}[scale=0.8]
    \def\m{4}
    \def\k{3}
    \def\n{12}
    \def\halfm{2} 

    \draw[gray!70, dotted, thick] (0,0) -- (\n,\n);

    \draw[magenta!60, thick] (0,0) 
        -- (1.5, 0.8) -- (2.8, 1.5) -- (4.5, 3.0) -- (5.8, 4.2) 
        -- (7.5, 6.0) -- (8.8, 7.5) -- (10.2, 9.0) -- (11.2, 10.5) 
        -- (\n,\n);

    \foreach \j in {0,1,2} {
        \pgfmathsetmacro{\cx}{\j*\m}
        \pgfmathsetmacro{\cy}{(\j+1)*\m}
        
        \fill[gray!20] (\cx, \cy) -- (\cx, \cy-\halfm) 
            -- (\cx+0.2, \cy-2.0) -- (\cx+0.8, \cy-1.1) 
            -- (\cx+1.1, \cy-0.8) -- (\cx+1.5, \cy-0.5) 
            -- (\cx+\halfm, \cy) -- cycle;
        
        \draw[cyan!60, thick] (\cx, \cy-\halfm) 
            -- (\cx+0.2, \cy-2.0) -- (\cx+0.8, \cy-1.1) 
            -- (\cx+1.1, \cy-0.8) -- (\cx+1.5, \cy-0.5) 
            -- (\cx+\halfm, \cy);
        
        \draw[cyan!60, thick] (\cx, \cy-2.6) 
            -- (\cx+0.2, \cy-2.1) -- (\cx+0.9, \cy-1.6) 
            -- (\cx+1.6, \cy-0.9) -- (\cx+2.1, \cy-0.2) 
            -- (\cx+2.6, \cy);
            
        \draw[cyan!60, thick] (\cx, \cy-3.2) 
            -- (\cx+0.3, \cy-2.6) -- (\cx+1.1, \cy-1.9) 
            -- (\cx+1.9, \cy-1.1) -- (\cx+2.6, \cy-0.3) 
            -- (\cx+3.2, \cy);
        
        \draw[magenta!60, thick] (\cx, \cy-\m) 
            -- (\cx+0.4, \cy-3.2) -- (\cx+1.2, \cy-2.5) 
            -- (\cx+2.5, \cy-1.2) -- (\cx+3.2, \cy-0.4) 
            -- (\cx+\m, \cy);
    }

    \foreach \j in {1,2} {
        \pgfmathsetmacro{\cx}{\j*\m}
        \pgfmathsetmacro{\cy}{\j*\m}
        
        \fill[gray!20] (\cx, \cy) -- (\cx-\halfm, \cy) 
            -- (\cx-1.9, \cy+0.2) -- (\cx-1.1, \cy+0.8) 
            -- (\cx-0.8, \cy+1.1) -- (\cx-0.4, \cy+1.5) 
            -- (\cx, \cy+\halfm) -- cycle;
        
        \draw[cyan!60, thick] (\cx-\halfm, \cy) 
            -- (\cx-1.9, \cy+0.2) -- (\cx-1.1, \cy+0.8) 
            -- (\cx-0.8, \cy+1.1) -- (\cx-0.4, \cy+1.5) 
            -- (\cx, \cy+\halfm);
        
        \draw[cyan!60, thick] (\cx-2.6, \cy) 
            -- (\cx-2.1, \cy+0.2) -- (\cx-1.6, \cy+0.9) 
            -- (\cx-0.9, \cy+1.6) -- (\cx-0.2, \cy+2.1) 
            -- (\cx, \cy+2.6);
            
        \draw[cyan!60, thick] (\cx-3.2, \cy) 
            -- (\cx-2.6, \cy+0.3) -- (\cx-1.9, \cy+1.1) 
            -- (\cx-1.1, \cy+1.9) -- (\cx-0.3, \cy+2.6) 
            -- (\cx, \cy+3.2);
    }

    \draw[black, very thick] (0,0)
    \foreach \j in {0,1,2} {
        -- (\j*\m, \j*\m+\m) -- (\j*\m+\m, \j*\m+\m)
    };

    \draw[gray!80, thick] (0,0) -- (\n+1.5, 0);
    \draw[gray!80, thick] (0,0) -- (0, \n+1.5);

    \fill[black] (0,0) circle (2.5pt) node[below left] {$(0,0)$};
    \fill[black] (\n,\n) circle (2.5pt) node[above right] {$(n,n)$};

\end{tikzpicture}

%% file: fig_staircase_notation.tex
\begin{tikzpicture}[scale=1.2, every node/.style={font=\large}]

    \draw[gray!70, dotted, thick] (-2,-2) -- (9,9);

    \draw[black, very thick] (0,0) -- (0,4) -- (4,4) -- (4,8);

    \draw[black, very thick] (-1.5, 0) -- (0,0);
    \fill[black] (-1.0, 0) circle (1.2pt);
    \fill[black] (-0.5, 0) circle (1.2pt);
    \node[left] at (-1.5, 0) {\dots};
    
    \draw[black, very thick] (4,8) -- (6.5, 8);
    \fill[black] (4.5, 8) circle (1.2pt);
    \fill[black] (5.0, 8) circle (1.2pt);
    \fill[black] (5.5, 8) circle (1.2pt);
    \fill[black] (6.0, 8) circle (1.2pt);
    \node[right] at (6.5, 8) {\dots};

    \foreach \y in {0.0, 0.5, 1.0, 1.5, 2.0, 3.0, 3.5, 4.0} { \fill[black] (0, \y) circle (1.2pt); }
    \foreach \x in {0.0, 0.5, 1.0, 2.0, 3.0, 3.5, 4.0} { \fill[black] (\x, 4) circle (1.2pt); }
    \foreach \y in {4.0, 4.5, 5.0, 6.0, 6.5, 7.0, 7.5, 8.0} { \fill[black] (4, \y) circle (1.2pt); }

    
    \fill[black] (0, 2.5) circle (1.5pt) node[left=6pt] {$u^i_j$};
    \fill[black] (0, 4) circle (1.5pt) node[above left=6pt] {$u^Q_j = v^Q_j$};
    \fill[black] (1.5, 4) circle (1.5pt) node[above=6pt] {$v^i_j$};

    \draw[cyan!80!blue, line width=3.5pt, rounded corners=2pt] 
        ([xshift=-4.5pt]0, 2.5) -- node[left, black] {$D^i_j$} ([xshift=-4.5pt, yshift=4.5pt]0, 4) -- ([yshift=4.5pt]1.5, 4);

    \draw[->, red, thick, decorate, decoration={snake, amplitude=0.5mm, segment length=3mm}] 
        (0, 2.5) -- (1.5, 4) node[midway, below right, black] {$L^i_j$};

    
    \fill[black] (2.5, 4) circle (1.5pt) node[below=6pt] {$v^{i'}_j$};
    \fill[black] (4, 4) circle (1.5pt) node[below right=6pt] {$v^0_j = u^0_{j+1}$};
    \fill[black] (4, 5.5) circle (1.5pt) node[right=6pt] {$u^{i'}_{j+1}$};

    \draw[cyan!80!blue, line width=3.5pt, rounded corners=2pt] 
        ([yshift=-4.5pt]2.5, 4) -- node[below, black] {$\tilde{D}^{i'}_j$} ([xshift=4.5pt, yshift=-4.5pt]4, 4) -- ([xshift=4.5pt]4, 5.5);

    \draw[->, red, thick, decorate, decoration={snake, amplitude=0.5mm, segment length=3mm}] 
        (2.5, 4) -- (4, 5.5) node[midway, above left, black] {$\tilde{L}^{i'}_j$};

\end{tikzpicture}

%% file: fig_area.tex
\begin{tikzpicture}[scale=1.0, >=stealth, very thick]
    \colorlet{pathgamma}{blue!60}
    \colorlet{pathpi}{red!60}
    \colorlet{posregion}{green!10}
    \colorlet{negregion}{red!10}

    \draw[step=1cm, gray!30, thin] (0,0) grid (6,6);

    \fill[negregion, opacity=0.7] (0,0) -- (0,2) -- (1,2) -- (1,0) -- cycle;
    \node at (0.5, 1) {\Large $-$};

    \fill[posregion, opacity=0.7] (1,2) -- (2,2) -- (2,3) -- (1,3) -- cycle;
    \node at (1.5, 2.5) {\Large $+$};

    \fill[negregion, opacity=0.7] (2,3) -- (2,6) -- (6,6) -- (6,3) -- cycle;
    \node at (4, 4.5) {\Huge $-$};

    \draw[->, pathpi, dashed] (6,6) -- (6,3) node[midway, right, text=black] {$\pi$};
    \draw[pathpi, dashed] (6,3) -- (2,3);
    \draw[->, pathpi] (2,3) -- (1,3);
    \draw[pathpi] (1,3) -- (1,2);
    \draw[->, pathpi, dashed] (1,2) -- (1,0);
    \draw[->, pathpi, dashed] (1,0) -- (0,0) ;

    \draw[->, pathgamma] (0,0) -- (0,2) node[midway, left, text=black] {$\gamma$};
    \draw[pathgamma] (0,2) -- (1,2);
    \draw[->, pathgamma, dashed] (1,2) -- (2,2);
    \draw[->, pathgamma, dashed] (2,2) -- (2,3);
    \draw[->, pathgamma] (2,3) -- (2,6);
    \draw[->, pathgamma] (2,6) -- (6,6);

    \filldraw[fill=white, draw=black, thick] (0,0) circle (2pt) node[below left, text=black] {$(0,0)$};
    \filldraw[fill=white, draw=black, thick] (1,2) circle (2pt);
    \filldraw[fill=white, draw=black, thick] (2,3) circle (2pt);
    \filldraw[fill=white, draw=black, thick] (6,6) circle (2pt) node[above right, text=black] {$(n,n)$};

\end{tikzpicture}